\documentclass[11pt,fleqn]{amsart}

\usepackage{epsfig}
\usepackage{amsmath}
\usepackage{amssymb}
\usepackage{amscd}
\usepackage{latexsym}
\usepackage{tabularx}
\usepackage{a4wide}
\usepackage[usenames]{color}
\usepackage{enumerate}
\usepackage{subfigure}
\usepackage{url}

\usepackage{url}

\usepackage[matrix,arrow,tips,curve]{xy}
\usepackage{pb-diagram,pb-xy}
\usepackage{verbatim}
\usepackage{mathrsfs}
\usepackage{color}

\numberwithin{equation}{section}

\newtheorem{thm}{Theorem}[section]

\newtheorem{prop}[thm]{Proposition}

\newtheorem{claim}[thm]{Claim}
\newtheorem{cor}[thm]{Corollary}

\newtheorem{remark}[thm]{Remark}
\newtheorem{defn}[thm]{Definition}

\newtheorem{question}[thm]{Question}

\renewcommand{\r}{\mathfrak r}

\makeatletter
\DeclareRobustCommand{\cev}[1]{%
  \mathpalette\do@cev{#1}%
}
\newcommand{\do@cev}[2]{%
  \fix@cev{#1}{+}%
  \reflectbox{$\m@th#1\vec{\reflectbox{$\fix@cev{#1}{-}\m@th#1#2\fix@cev{#1}{+}$}}$}%
  \fix@cev{#1}{-}%
}
\newcommand{\fix@cev}[2]{%
  \ifx#1\displaystyle
    \mkern#23mu
  \else
    \ifx#1\textstyle
      \mkern#23mu
    \else
      \ifx#1\scriptstyle
        \mkern#22mu
      \else
        \mkern#22mu
      \fi
    \fi
  \fi
}

\makeatother

\DeclareMathOperator\val{val}

\newcommand{\br}{\mathrm{br}}

\title{Logarithmic Tree Factorials}

\author{Omid Amini}
\address{CNRS - D\'epartement de math\'ematiques et applications, \'Ecole Normale Sup\'erieure, Paris}
\email{oamini@math.ens.fr}
\begin{document}
\begin{abstract}
 To any rooted tree, we associate a sequence of numbers  that we call the logarithmic  factorials of the tree. This provides a generalization of Bhargava's factorials to a natural combinatorial setting suitable for studying questions around generalized factorials.

  We discuss several basic aspects of the framework in this paper. In particular, we relate the growth of the sequence of logarithmic factorials associated to a tree to the transience of the random walk and the existence of a harmonic measure on the tree, obtain an equidistribution theorem for factorial-determining-sequences of subsets of local fields, and provide a factorial-based characterization of the branching number of infinite trees. 

 Our treatment is based on a local weighting process in the tree which gives an effective way of  constructing the factorial sequence.

\end{abstract}
\maketitle

\section{Introduction}

Let $T$ be a rooted tree with root $\r$, and let $\ell: E(T) \rightarrow \mathbb R_{+}$ be a length function on the edges of $T$. 
Denote by $\Gamma$ the metric realization of the pair $(T, \ell)$, which is a rooted metric tree with root $\r$. We call the unite length function $\ell \equiv 1$ which assigns value one to all the edges of a tree $T$ the \emph{standard} length function. 

We orient $T$ away from the root, and, by an abuse of the notation, denote by $E(T)$ the set of oriented edges of $T$.  For any vertex of $T$, we denote by $[\r,v]$ the oriented path (resp. segment) from $\r$ to $v$ in $T$ (resp. $\Gamma$).

Consider the boundary $\partial T$ of $T$, which is by definition, the set of all infinite oriented paths  in $T$ with starting vertex at the root $\r$, and define the \emph{extended boundary} $\widetilde \partial T$ as the union of $\partial T$ with the set of all oriented paths in $T$ from the root $\r$ to a leaf of $T$. For any pair $(T, \ell)$ with metric realization $\Gamma$, define $\widetilde \partial (T, \ell) = \widetilde \partial \Gamma = \widetilde \partial T$. 
For any point $\rho \in \widetilde \partial T$, we denote by $E(\rho)$ the set of all the edges of $T$ which are in $\rho$.

\medskip

Any two different elements of $\partial T$ have a finite number of edges in common. So we can define a non-negative real-valued \emph{intersection pairing} $\langle\,,\rangle$ on $\widetilde \partial \Gamma$ as follows. For any two points $\rho, \tau \in \widetilde \partial \Gamma$, with $\rho \neq \tau$ if both $\rho $ and $\tau$ both belong to the boundary of $T$, let 
\[\langle \rho, \tau\rangle := \ell(\rho \cap \tau ) = \sum_{e \in E(\rho)\cap E(\tau)} \ell(e).\]

Consider the following \emph{greedy procedure} in choosing a sequence of elements $\rho_0, \rho_1, \dots$ in $\widetilde \partial \Gamma$.  Let $\rho_0 \in \widetilde \partial \Gamma$ be any arbitrary element of the extended boundary. Proceeding inductively on $n\in \mathbb N$, assume that $\rho_0, \dots, \rho_{n-1} \in \widetilde \partial \Gamma$ have been chosen, and choose $\rho_{n}$, if possible, arbitrarily among the set of all elements $\rho \in \widetilde \Gamma \setminus \{\rho_0, \dots, \rho_{n-1}\}$ which minimizes the sum $\sum_{j=0}^{n-1} \langle \rho, \rho_j\rangle$. 
Define $a_n:= \sum_{j=0}^{n-1} \langle \rho_{n},\rho_j\rangle.$  We have 
\begin{thm} \label{thm:main1-intro} For any pair $(T, \ell)$ consisting of a rooted tree $T$ and a length function $\ell$ on $T$, the sequence $\{a_n\}$ constructed above only depends on the metric realization $\Gamma$ of $(T, \ell)$.
\end{thm}
We call the number $a_n$ both the $(T,\ell)$ and $\Gamma$\emph{-factorial} of $n$, and denote it by $ n!_{(T,\ell)}$ or $n!_\Gamma$. We call the sequence $\{\rho_n\}$ in the construction above a \emph{factorial-defining sequence} for $(T, \ell)$ and $\Gamma$. When $\ell$ is the standard length function, we simple write $n!_T$ for the factorials of the pair $(T, \ell)$. 
\medskip

The above definition is a direct extension to arbitrary (metric) trees of the (logarithmic) factorial sequence associated by Bhargava to subsets of the ring of valuation of a local field, that we now recall~\cite{Bha97, Bha09}.

 Let $K$ be a local field with discrete valuation $\mathrm{val}$, with ring of valuation $R$, with maximal ideal $\mathfrak m$, and with residue field $\kappa = R/ \mathfrak m$, which is thus a finite field.  Let $S$  be a subset of $R$. The logarithmic factorial sequence associated to $S$ is obtained as follows. Choose $s_0 \in S$ arbitrary. Proceeding inductively, and assuming $s_{0}, \dots, s_{n-1}$ are already chosen, choose $s_{n}$ among all $s\in S$ which minimizes the quantity $\mathrm{val}\bigl(\,\prod_{j=0}^{n-1} (s-s_j)\,\bigr)$. 
 Define \[n!_S := \val\bigl(\prod_{j=0}^n (s_{n}-s_0) \dots (s_{n} - s_{n-1})\bigr).\]

To any subset $S \subset R$ of $K$ as above, one can associate its \emph{adelic tree} $T_S$, which is a rooted locally finite tree with vertices of valence bounded by $|\kappa|+1$, as follows. For each integer $h\in \mathbb N_*$, consider the projection $\phi_h: \, R \rightarrow R/\mathfrak m^h$, and define $V_h = \phi_h(S)$.  The rooted tree $T_S$ has vertex set 
$\sqcup_{h=0}^{\infty} V_h$, and has as root the unique element of $V_0$. The edge set of $T_S$ is defined as follows. For any $h$, there exists a map $\pi_h: R/\mathfrak m^{h+1} \rightarrow R/\mathfrak m^{h}$, and we have 
$\phi_h = \pi_h\circ \phi_{h+1}$. A vertex $u$ in $V_h$ is adjacent to a vertex $v\in V_{h+1}$ if and only if $\pi_h(v) =u$. 

In the case $S = R$, the tree $T_R$ is the $|\kappa|$-regular tree, and obviously, for any subset $S \subset R$, the tree $T_S$ is a subtree of $T_R$. Consider the closure $\overline S$ of $S$ in $K$. The elements of $\widetilde \partial T_S$, viewed in $\overline S$, form a dense subset $S_0$ of $\overline S$, and the factorials of the tree $T_S$, as defined above, correspond to the factorials of the subset $S_0\subset R$. Since the factorials of $S_0$, $\overline S$, and $S$ are all equal, c.f.~\cite{Wood03}, we get the following proposition.

\begin{prop} Let $K$ be a local field with valuation ring $R$. Let $S$ be  a subset of $R$ with adelic tree $T_S$. We have $n!_S = n!_{T_S}$, where $n!_S$ denotes the  Bhargava's $S$-factorial of $n$.  
\end{prop}

The proof given by Bhargava of the well-definedness of the factorial sequence $n!_S$ is indirect and goes through the ring of integer valued polynomials on $S$. In order to prove Theorem~\ref{thm:main1-intro}, we give an alternative local definition of a sequence associated to a pair $(T, \ell)$, show by induction that it is well-defined and only depends on the metric realization $\Gamma$, and then prove the equivalence of that definition with the definition given above.  Thus our proof leads to an alternative combinatorial proof of the well-definedness of the factorial sequence associated to a  subset of local fields. 

\medskip

Note that  we have not  made so far any finiteness assumption on the valence of vertices of $T$.  In fact, as we will explain in a moment, we can always reduce to the case of locally finite trees with a \emph{capacity function} on leaves, so we next define such objects.

\subsection{Locally finite trees with a capacity function on leaves} \label{sec:locfin-intro}
Let $T$ be a locally finite rooted tree and let $\ell$ be a length function on $E(T)$. Denote by $L(T)$ the set of all  leaves of $T$. By a \emph{capacity function} on $T$ we mean a function $\chi: L(T) \rightarrow \mathbb N \cup\{\infty\}$. We modify the definition of the factorial sequence given in the previous section by taking into account the capacity of leaves of $T$ as follows.  Assuming for an integer $n\in \mathbb N$ that $\rho_0, \dots, \rho_{n-1}$ are chosen, we choose $\rho_n$, if possible, among those $\rho \in \widetilde \partial \Gamma$ which minimizes the sum $\sum_{j=0}^{n-1} \langle \rho, \rho_j\rangle$, and which verify the capacity condition that, when $\rho$ is a leaf of $T$, the number of times $\rho$ appears in the sequence $\rho_0, \dots, \rho_{n-1}$ is strictly less than the capacity of $\rho$. So in the sequence $\rho_0, \rho_1, \dots$ each leaf of $T$ can appear at most as many times as its capacity. We define 
\begin{equation}\label{def:fac-intro}
a_{n} := \sum_{j=0}^{n-1} \langle \rho_{n}, \rho_j\rangle.
\end{equation}
 Then we have the following Theorem.
\begin{thm} \label{thm:main1bis-intro} The sequence $\{a_n\}$ only depends on the pair $(\Gamma, \chi)$, where $\Gamma$ is the metric realization of the pair $(T, \ell)$.
\end{thm}
We call $a_n$ the $(\Gamma,\chi)$ or $(T, \ell, \chi)$-factorial of $n$, and denote it by $n!_{(T, \ell,\chi)} = n!_{(\Gamma, \chi)}$. When $\ell$ is the standard length function, we simply write $n!_{(T, \chi)}$.

\medskip

Let $S$ be a subset of the valuation ring $R$ of a local field $K$. Let $h\in \mathbb N$. In the adelic tree $T_S$ of $S$ consider the subtree $T_{S, h}$ of all the vertices at distance at most $h$ from the root $\r$ of $T_S$. Define the capacity function $\chi_h$ on leaves of $T_{S, h}$ as follows. For any leaf $v$ of $T_{S, h}$, consider the subtree $T_{S, v}$ of $T$ which consists of $v$ and all its descendants, and define $\chi_h(v)$ as the number of elements in the extended boundary of $T_{S, v}$. We have the following  proposition.
\begin{prop} 
Notations as above, we have $n!_{(T, \chi_h)} = n!_{S, h}$
\end{prop}
Thus, the factorials in the presence of a capacity function generalizes  factorials of order $h$ for subsets of local fields in the terminology of~\cite{Bha09}. 
\subsection{Reduction to locally finite trees} \label{sec:reduction}
Let $(T, \ell)$ be  a pair consisting of a tree $T$ and a length function $\ell$ on $T$. We define  the \emph{locally finite component} $T_0$ of $T$ as follows. Consider the set $V_0$ of all vertices $v$ of $T$ with the property that all the interior vertices of the oriented path $[\r, v]$ have bounded valence in $T$. So, for example,  if the root $\r$ has infinite valence, then $V_0$ consists of a single vertex $\r$. Define the subtree $T_0$ of $T$ as the tree induced by $T$ on  $V_0$. For any leaf of $T_0$ which is a vertex of valence infinity in $T$, define the capacity  $\chi_0(v)$ of $v$ to be infinity.  For other leaves of $T_0$, which are thus also leaves of $T$, define $\chi_0(v) =1$. Let $\ell_0$ be the restriction of $\ell$ to the edges of $T_0$. Then we have 
\begin{prop} \label{prop:reduction} Notation as above, we have for all $n$, $n!_{(T, \ell)}=n!_{(T_0, \ell_0, \chi_0)}.$
\end{prop}

Let  now $T$ be a locally finite tree, $\ell$ a length function on $T$, and $\chi$ a capacity function. Define the tree $T_1$ by adding $\chi(v)$ disjoint infinite paths to any leaf $v$ of $T$, and extend $\ell$ to a length function $\ell_1$ on $T_1$ by assigning arbitrary lengths to the new edges of $T_1$. It is easy to see that for any $n$, we have $n!_{(T, \ell, \chi)} = n!_{(T_1, \ell_1)}$. 

 Therefore, in what follows, there is no restriction in assuming the tree $T$ is locally finite, and, if necessary, a capacity function $\chi$ is given. 

\subsection{Growth of the factorial sequence and equidistribution}
 Let $T$ be a locally finite rooted tree and $\ell$ be a length function on $T$. 
 Denote by $\Gamma$ the metric realization of $(T, \ell)$. 
 
We will prove that 
\[\forall m,n\in \mathbb N, \qquad (m+n)!_\Gamma \geq m!_\Gamma + n!_\Gamma. \]

Combining this with Fekete's lemma, we get the convergence of the sequence 
\[\frac 1n n!_\Gamma \rightarrow H(\Gamma).\] 
The quantity $H(\Gamma)$, that we deliberately denote by $H(T, \ell)$ as well, is an invariant of $\Gamma$ and one of our objectives in this paper will be to characterize it. 

We first describe a necessary and sufficient condition for the finiteness of $H(T,\ell)$.
\medskip

Define the \emph{conductance} $c: E(T) \rightarrow \mathbb R_+$ given by $\forall uv\in E(T)$, $c(u,v) = c(v,u) := \frac 1{\ell(uv)}$.

Consider the random walk $RW(T, \ell)$ on $T$ which starts at the root $\r$, and which has probability of going from a vertex $u$ of the tree to any of its neighbors $v$ in the tree given by  $p_{uv} := \frac{c(u,v)}{\sum_{w\sim u} c(u,w)}$.

\medskip

The following theorem relates the finiteness of the limit of logarithmic factorials to the transience of the random walk on the tree.
\begin{thm}\label{equivalence-intro} Let $T$ be an infinite locally finite rooted tree and $\ell$  a length function on $T$. Assume that the pair $(T, \ell)$ is weakly complete. The following two statements are equivalent.
\begin{itemize}
\item The random walk $RW(T, \ell)$ is transient.
\item The limit $H(T, \ell)$ is finite.
\end{itemize}
Equivalently, the random walk $RW(T, \ell)$ is recurrent if and only if $H(T, \ell) = \infty$.
\end{thm}
The condition that $(T, \ell)$ is weakly complete means any infinite oriented path $P$ in $T$ which entirely consists of valence two vertices has to be of infinite length in $\Gamma$. In particular, this is the case if the length function is $\epsilon$-away from zero for some $\epsilon>0$. We refer to Section~\ref{sec:growth} for more details. 
 
 \medskip
 
In the presence of a capacity function $\chi$  on the leaves of $T$, the normalized factorials $\frac 1n n!_{(\Gamma, \chi)}$ still converge to a parameter $H(\Gamma, \chi)$, and the theorem above still holds if the values of $\chi$ are all finite as can be easily observed by the transformation $(T_1, \ell_1)$ of $(T, \ell, \chi)$ described in the previous section.  (Indeed, in this case, we will always have 
$H(T, \ell) = H(T, \ell, \chi)$.)
On the other hand, when $\chi$ takes value $\infty$ at some leaves of $T$, then the value of $H(T, \ell, \chi)$ is always finite. 

\medskip

We now turn to the question of determining the value of $H(T, \ell)$. By the previous theorem, we can assume that the random walk $RW(T, \ell)$ is transient. We have the following.

\begin{thm} \label{thm:main2-intro} Let $T$ be a locally finite tree and $\ell$ a length function on $T$ so that the random walk $RW(T, \ell)$ on $T$ is transient. Let $\eta$ be a the unit current flow on $T$ and $\mu_{\mathrm{har}}$ the corresponding harmonic measure on $\partial T$. Assume that $(T, \ell)$ is weakly complete. Then,
\begin{itemize}
\item any factorial determining sequence $\rho_0, \rho_1, \dots$ of $(T, \ell)$ is equidistributed in $\partial T$ with respect to the harmonic measure $\mu_{\mathrm{har}}$.
\item we have $H(T, \ell) = \|\eta\|^2$, where $\|\eta\|^2$ is the energy of the unit current flow $\eta$ on $T$. 
\end{itemize}
\end{thm}

As an immediate corollary, we get the following equidistribution theorem for factorial-determining sequences of subsets  of local fields. 
Let us call a subset $S$ of the valuation ring $R$ of a local field $K$ \emph{transient} if the adelic tree $T_S$ of $S$ is transient. For a transient subset $S$ of $R$, we denote by $\mu_{\mathrm{har}}$ the corresponding harmonic measure of $S$ which has support in the closure $\overline S$ of $S$ in $R$. We have

\begin{thm} Let $K$ be a local field with valuation ring $R$, and let $S$ be an infinite subset of $R$. The following two conditions are equivalent.
\begin{itemize}
\item The subset $S$ of $R$ is transient.
\item The sequence $\frac 1n n!_{S}$ converges to a finite $H(S) \in (0,\infty)$.
\end{itemize}
Moreover, for a transient subset $S$ of $R$, any factorial determining sequence $s_0,s_1, s_2, \dots$  of $S$ is equidistributed in $\overline S$ with respect to the harmonic measure, and we have 
\begin{align*}
H(S) &= \int_{\substack{(x,y) \in \overline S \times \overline S\\ x\neq y}}\val(x-y) d\mu_{\mathrm{har}}(x)d\mu_{\mathrm{har}}(y) \\
&=\int_{\overline S} \mathrm{val}(x_0-y) d\mu_{\mathrm{har}}(y) \qquad \textrm{a.s. for $x_0\in \overline S$}.
\end{align*}
 
\end{thm}

We note that in the presence of a capacity function on the leaves of $T$ which takes values infinity, the limit $H(T, \ell, \chi)$ has a similar expression. Indeed, it will be enough to consider the modified tree $T_1$ obtained by adding a countable number of paths to any leaf $v$ of $T$ with $\chi(v) =\infty$, and define the conductance of all these new edges to be equal to $\infty$. The random walk on $T_1$ with these  conductances is equivalent to a random walk on $T$ with absorption on  the leaves of capacity infinity, and the limit $H(T, \ell, \chi)$ is the squared norm of the unit current flow on $T_1$.  For the special case where $T$ is a finite tree and $\chi$ is a function on the leaves of $T$ which takes value infinity at some points of $T$, we have the following explicit way of calculating  $H(T,\ell, \chi)$. 

Let $L_0 \subset L(T)$ be the set of all leaves $v$ with $\chi(v)=\infty$, and define the connected graph $G = (V, E)$ obtained by identifying all the vertices in $L_0$ to a single vertex $\mathfrak s$. Let $C^0(G, \mathbb R)$ be the space of real valued functions on the vertices of $G$. The length function $\ell$ induces a length function on the edges of $G$,  to which we can associate a Laplcian operators $\Delta: C^0(G, \mathbb R) \to C^0(G, \mathbb R)$ as follows. For any function $f \in C^0(G, \mathbb R)$, the value of $\Delta(f) \in C^0(G, \mathbb R)$ at a vertex $v$ of $V(G)$ is given by
\[\Delta(f)(v) = \sum_{\{u,v\} \in E(G)} \frac 1{\ell_e} \Bigl(f(u) - f(v)\Bigr).
\] 
For a vertex $v$ of $G$, denote by ${\bf 1}_v$ the characteric function of $v$ which takes value one at $v$, and value zero outside $v$. Let $F$ be the real-valued function on $V$ which solves the Laplace equation $\Delta(F) = {\bf 1}_\r - {\bf 1}_{\mathfrak s}.$ By connectivity of $G$, up to additioning a constant function, $F$ is unique. 
\begin{thm}\label{thm:finite} Notations as above, we have $H(T, \ell, \chi) = F(\mathfrak s) - F(\r). $
\end{thm}
As an immediate corollary, for any $h \in \mathbb N$, we get a limit theorem for the factorials of order $h$ associated to subsets of local fields. 

\subsection{Factorial-based characterization of the branching number}
Let $T$ be a countable infinite locally finite tree. A cut-set in $T$ is a subset $C$ of vertices such that any infinite (oriented) part from root meet a vertex of $C$.  Recall that the \emph{branching number} of $T$ is defined as
\[\mathrm{br}(T) := \sup \Bigl\{ \lambda :\,\, \inf_{C\,\,\textrm{cut-set}}\sum_{v\in C}\lambda^{-|v|} > 0\,\Bigr\},\]
where for a vertex $v$, the distance of $v$ to $\r$ in the tree $T$ is denoted by $|v|$.

For any $\lambda>0$, denote by $\ell_\lambda$ the length function on $T$ which associates the length $\lambda^{|u|}$ to any oriented edge $uv \in E(T)$. 
The following classical theorem relates the transience of the simple random walk on $(T, \ell_\lambda)$ with the branching number of $T$.
\begin{thm}[R. Lyons~\cite{Lyons90}] The random walk $RW_{(T, \ell_\lambda)}$ on $T$ is transient provided that $\lambda< \mathrm{br}(T)$. For $\lambda> \br(T)$, the random walk $RW_{(T, \ell_\lambda)}$ is recurrent.
\end{thm}
Combining this with our Theorem~\ref{equivalence-intro}, we get the following characterization of the branching number in terms of tree factorials.
\begin{thm} The branching number of $T$ is the supremum of $\lambda$ so that the normalized factorials of the pair $(T,\ell_\lambda)$ have a finite limit.
\end{thm}

\subsection{Organization of the paper} The local weighting process and the proofs of Theorem~\ref{thm:main1-intro}, Theorem~\ref{thm:main1bis-intro}, and Propositios~\ref{prop:reduction}, as well as basic properties of the weighting process and the tree factorials, are presented in Section~\ref{sec:def}. In Section~\ref{sec:hminf} we consider the important question of how much information the factorial sequence gives about the tree. We show the realizability of any sufficiently biased sequence of non-negative numbers as the factorials of a pair $(T, \ell)$, and 
deduce from that construction, the existence of different trees with the same factorial sequence. 

The growth of the factorial sequence is studied in Section~\ref{sec:growth}. The equivalence Theorem~\ref{equivalence-intro}, as well as the limit and equidistribution theorem~\ref{thm:main2-intro} are proved in that section. We omit the proof of Theorem~\ref{thm:finite}, which can be obtained by the same arguments.  

Some concluding remarks are given in Section~\ref{sec:concluding}.

\section{Definition and  basic properties}\label{sec:def}
In this section, we give the definition of the factorial sequence in terms of a \emph{local exploration process} in the tree. We then show later that this definition is equivalent to the definition given in the introduction.

Let $T$ be a locally finite rooted tree. Denote by $\r$ the root of $T$. We orient $T$ away from the root. For any vertex $v$ there is a unique oriented path from $\r$ to $v$ that we denote by $[\r,v]$, and denote by $|v|$ the length of $[\r, v]$ which we call the \emph{generation} of $v$. We write $u \leq v$ if $u$ lies in the oriented path from 
 $\r$ to $v$. The \emph{parent} of a vertex $v \neq \r$ is the unique vertex $u$ with $u\leq v$, and $|u| = |v|-1$; 
 it is denoted by $\cev{v}$.  For two vertices $u,v$, we write $u \sim v$ if $u$ and $v$ are adjacent in the tree. The valence of a vertex $v$ is the number of vertices $u\sim v$ in the tree. 

 Edges in this paper mean oriented edges, the edge $uv$ is thus oriented 
 from $u$ toward $v$, and we have $u = \cev v$.  If a vertex $v$ is a descendant of another vertex $u$, we denote by $[u,v]$ the unique path from $u$ to $v$. A \emph{pending} edge of a vertex $u$ is an edge which joins $u$ to one of its children.  For any vertex $v$ in $T$, we denote by $\br(v)$ the number of children of $v$.  A vertex $v$ in $T$ with $\br(v) \geq 2$ is called \emph{branching}. The  set of all the branching vertices of $T$ is denoted by $\mathscr B(T)$. 

\medskip

By a \emph{leaf} of a rooted tree $T$ we mean any vertex $v \neq \r$ of $T$ of valence one if $T$ is not reduced to a single vertex $\r$. Otherwise, if $T$ has a unique vertex $\r$, then $\r$ is a leaf of $T$.  We denote by $L(T)$ the set of all the leaves of $T$. 

An \emph{internal} vertex of $T$ is  any vertex different from  the leaves of $T$.

 \medskip
 
By a \emph{strict path} in $T$ we mean any oriented path $P$ which starts from the root $\r$, does not contain any branching vertex in its interior, and which is maximal with respect to this property (for the inclusion of paths). It follows that a strict path $P$ 
 either connects $\r$ to a branching vertex of $T$, or connects $\r$ to a leaf of $T$, or is infinite and $T\setminus \Bigl(P\setminus\{ \r\}\Bigr)$ is connected. In addition, for any pending edge  $\r u$ at $\r$, there exists a unique strict path which contains $u$, and these are all the strict paths of $T$.

\medskip

For any vertex $v$ in $T$, we denote by $T_v$ the subtree of $T$ consisting of $v$ and all of its descendants, rooted at $v$. 

\medskip

Let $\ell: E(T) \to \mathbb R_+$ be a \emph{length function} on the edges of $T$ which assigns to each edge  $e$ of $T$ its length $\ell_e = \ell(e)$. Denote by $\Gamma$ the metric tree associated to the pair $(T, \ell)$. Recall that $\Gamma$ is the disjoint union of the vertex set $V(T)$ and open intervals $I_e$ of length $\ell_e$, for $e\in E$, with the identification of the end-points of $I_e$ with the corresponding vertices in $V(T)$. The pair $(T, \ell)$ is called a \emph{model} of $\Gamma$.

A great source of examples for what follows are trees coming from an arithmetic situation, in which case the length function $\ell$ is the constant function $1$. We call the constant length function $1$ the \emph{standard} length function.

\medskip

A \emph{capacity} function on $T$ is any function $\chi: L(T) \rightarrow \mathbb N \cup\{\infty\}$ giving a capacity to any leaf of $T$. The standard capacity function is the constant function $1$, and if there is no mention of the capacity  in what follows, it means all the leaves have capacity one. 
\medskip

A \emph{weighted tree} $(T, \omega)$ in this paper means a tree $T$ with a weight function $\omega: E(T) \rightarrow \mathbb N \cup\{\emptyset\}$, such that the set of edges $e$ with $\omega(e) \neq \emptyset$ forms a connected subgraph of $T$. 
An edge $e$ of $(T, \omega)$ with $\omega(e) =\emptyset$ is called \emph{unweighted}; all the other edges are called \emph{weighted}.

\medskip

For a weighted tree $(T, \omega)$, we denote by $T_\omega$ the subtree of $T$ which contains the root and all the weighted edges $e\in E(T)$. A vertex $u$ of $T$ is called \emph{clear} if all the pending edges of $u$ are unweighted. 
The clear vertices are all the vertices of $T$ which are either a leaf of $T_\omega$ or does not belong to $T_\omega$. 

A vertex $v$ of $T_\omega$ is called \emph{unsaturated} if either  $v$ is an internal vertex of $T_\omega$ and there is an edge in $E(T) \setminus E(T_w)$ incident to $v$, or, in the presence of a capacity function on $T$,  $v$ is a leaf of $T$ and $\omega(\cev v v) < \chi(v)$.

For a tree $T$ with a length function $\ell$ and weight function $\omega$,   the \emph{weighted length} $\ell_\omega$ of a path $P$  in $T_\omega$ is defined as 
\[\ell_{\omega}(P) := \sum_{e\in E(P)} \omega(e)\ell_e.\]

We now describe a \emph{weighting process} which provides an alternative equivalent definition of the factorial sequence. 

 Let $T$ be a locally finite tree, $\ell$ a length function on $T$, and $\chi$ a capacity function on $L(T)$.   Consider the weight function $\omega_0$ which assign $\emptyset$ to any edge. Let $T_0:=T_{\omega_0}$, and note that all the edges of $T$ are unweighted and we have $T_0= \{\r\}$.   We recursively 
construct a sequence of weighted trees $(T, \omega_n)$ and a sequence of non-negative real numbers $a_n$. The construction will be so that for all $n\geq 1$

\begin{center}
$(*)$\qquad \emph{ all the clear vertices of $T_{\omega_n}$ are either branching or a leaf in $T$.}
\end{center}

Define $a_0:=0$. Let $\r v$ be any edge of $T$ incident to $\r$. If such an edge does not exist, i.e., if $T$ is reduced to a single vertex $\r$, then we stop, and define $x_n = \r$ and $a_n = 0$ for all $1\leq n <\chi(\r)$. Otherwise, 
let $P$ be a strict path in $T$ containing both $\r$ and $v$. Define 
\[\omega_{1} (e) :=   
 \begin{cases}
    1 \qquad & \textrm{for all edges} \,\, e \in P \\
     \omega_0(e) =\emptyset \qquad & \textrm{otherwise}.
   \end{cases}    \]
Note that $(*)$ is clearly verified for $T_{w_1}$. 

Proceeding by induction, suppose now that $n\in \mathbb N$, we are at stage $n$ and we have a weighted tree $(T, \omega_n)$, a sequence of vertices $x_1, \dots, x_{n-1}$ of $T$, and a sequence of integers $a_0, \dots,a_{n-1}$. Let $T_n :=T_{\omega_n}$ and denote by $U_n$ the set of all the unsaturated vertices $v$ of $T_n$. 
Assume that $T_{n}$ verifies $(*)$.

If $U_n = \emptyset$, then we stop. Otherwise, if $U_n$ is non-empty, choose a vertex $x_n$ in $U_n$ with minimum weighted length to the root $\r$, i.e., so that $\ell_{\omega_n} ([\r, x_n]) = \min_{v\in U_n} \ell_{\omega_n}(\r, v)$. 
Set $a_n := \ell_{\omega_n}(x_n)$, and 
 define the weighted tree $(T, \omega_{n+1})$ as follows, depending on whether
 $x_n$ is clear or not.

\medskip

\begin{itemize}
\item[(1)] Either $x_n$ is not clear. In this case, choose a pending edge $x_ny$  at $x_n$ with $\omega_n(x_ny)=\emptyset$. Let  $P_y$ be the unique strict path in the subtree $T_{x_n}$ which contains $y$. Define
\[\omega_{n+1} (e) :=   
 \begin{cases}
    1 \qquad & \textrm{if} \,\, e \,\, \textrm{belongs to $P_y$}, \\
     \omega_n(e)+1 \qquad &\textrm{if} \,\, e \,\, \textrm{belongs to the path}\,\, [\r, x_n] \,\, \textrm{in} \,\, T,\\
     \omega_n(e) \qquad & \textrm{otherwise}.
   \end{cases}    \]

\item[(2)] Or $x_n$ is a clear vertex of $T_{n}$. By Property $(*)$, $x_n$ is either branching in $T$ or it belongs to $L(T)$. 

\begin{itemize}
\item[(2.1)] If $x_n$ is branching, then choose any two pending edges $e_1=x_nz$ and $e_2=x_nw$ at $x_n$, and consider the two (disjoint) strict 
paths $P_z$ and $P_w$ in  $T_{x_n}$ with $z\in P_z$ and $w\in P_w$.  
 Define 
 \[\omega_{n+1} (e) :=   
 \begin{cases}
    1 \qquad & \textrm{if} \,\, e \,\, \textrm{belongs to the union} \,\, P_z\cup P_w, \\
     \omega_n(e)+1 \qquad &\textrm{if} \,\, e \,\, \textrm{belongs to the path}\,\, [\r,x_n],\\
     \omega_n(e) \qquad & \textrm{otherwise}.
   \end{cases}    \]
   \item[(2.2)] If $x_n$ is a leaf of $T$, then since $x_n$ is unsaturated, we have $\omega_n(\cev x_n x_n )< \chi(x_n)$. Define 
    \[\omega_{n+1} (e) :=   
 \begin{cases}
     \omega_n(e)+1 \qquad &\textrm{if} \,\, e \,\, \textrm{belongs to the path}\,\, [\r,x_n],\\
     \omega_n(e) \qquad & \textrm{otherwise}.
   \end{cases}    \]
   \end{itemize}
   \end{itemize}

   Let $T_{n+1} := T_{\omega_{n+1}}$. Any clear vertex $v$ of $T_{n+1}$ is either a clear vertex of $T_n$, or an end-point of a strict path in the subtree $T_{x_n}$ (one among $P_y, P_z, P_w$). It follows that $v$ is either branching or a leaf in $T$. Thus, $T_{n+1}$ verifies Property $(*)$, and the above definition results in a sequence of weighted trees $(T, \omega_i)$, a sequence of vertices $x_i$, and specially, a sequence of reals $a_i$.
   
    Note  that in the case the length function takes integer values, all the numbers $a_i$ are integers.  
   
   By definition, it is easy to see that in the case $\chi$ is the standard capacity function, the sequence is infinite if and only if the number of branching vertices of $T$ is infinite. More generally, define $N_{T,\chi}$ by
    \begin{equation}\label{form:nt}
  N_{T, \chi} := 1+ \sum_{v\in \mathscr B} (\br(v)-1) + \sum_{v\in L(T)} (\chi(v)-1).
 \end{equation} 
 One can see directly from the definition that  $(T, \omega_n)$ and $a_n$ are defined provided that $0\leq n< N_{T, \chi}$, as in this case $U_n$ is always non-empty. (See also the proof of Theorem~\ref{thm:main1} below).

The sequence $(T, \omega_n)$ is obviously not unique in general as it involves making a choice of a vertex $x_n\in U_n$ and strict paths in some subtrees at each stage. However, the sequence $\{a_i\}_{0\leq i<N_{T,\chi}}$ only depends on $(T, \ell)$ (actually, only on the rooted metric tree $\Gamma$ associated to $(T, \ell)$). 

\begin{thm}\label{thm:main1}
 $(i)$ The sequence $a_0,a_1, a_2, \dots$ only depends on $(T,\ell, \chi)$. 
 
\noindent  $(ii)$ For two pairs $(T_1, \ell_1)$ and $(T_2, \ell_2)$ with the same metric realization, any capacity function $\chi$ on $T_1$ induces a capacity function on $T_2$, and the triples $(T_1, \ell_1, \chi)$ and $(T_2, \ell_2, \chi)$ have the same factorial sequence. 
\end{thm}
This leads to the following definition.

\begin{defn}[Logarithmic tree factorials]\rm
$\bullet$ Let $T$ be a rooted locally finite tree.  For each integer $0\leq n < N_T$, the integer $a_n$ associated to the tree $T$ with standard length function $\ell \equiv 1$ is called the $T$-factorial of $n$ and is denoted by $n!_T$.    

$\bullet$ Let $\Gamma$ be a rooted metric tree with a model $(T, \ell)$, where $T$ is a rooted locally finite tree and $\ell$ a length function on $E(T)$. Let $\chi : L(T) \rightarrow \mathbb N \cup\{\infty\}$ be a capacity function. For each integer $0\leq n < N_{T,\chi}$, the real number $a_n$ associated to the tree $T$ with length function $\ell$ and with capacity $\chi$ is called the $(T,\ell,\chi)$-factorial or $(\Gamma,\chi)$-factorial of $n$ and is denoted by $n!_{(T,\ell,\chi)} = n!_{(\Gamma,\chi)}$. If $\chi$ is the standard capacity function, we drop $\chi$ and simply write $n!_{(T,\ell)}$ or $n!_{\Gamma}$.

$\bullet$ For any pair $(T, \ell)$ with metric realization $\Gamma$ and a capacity function $\chi$, the sequence $x_n$ in the weighting process described above is a called a \emph{factorial-determining} or \emph{factorial-defining} sequence for $(T, \ell,\chi)$ and $(\Gamma,\chi)$. The sequence of trees $T_n$ in the weighting process is called a sequence of \emph{factorial trees} for $(T, \ell)$.
\end{defn}
Note that neither the factorial-defining sequence nor the factorial trees are unique in general. 

\begin{proof}[Proof of Theorem~\ref{thm:main1}$(i)$]
 We proceed by induction. Consider the following property $\mathscr P_n$: 
 \begin{center}
  $(\mathscr P_n)$ \qquad For any locally finite tree $T$, any length function $\ell: E(T) \rightarrow \mathbb R_+$, any capacity function $\chi$ on $T$, and any $0\leq i \leq \min\{n,N_{T,\chi}-1\}$, the number $a_i$ only depends on $(T, \ell, \chi)$.
 \end{center}

By our definition, $a_0=0$ for any tree $T$, so obviously $\mathscr P_0$ holds. We show that $\mathscr P_n$ implies $\mathscr P_{n+1}$, from which the theorem follows.

 \medskip
 
Assume $\mathscr P_n$ holds. Let $T$ be any locally finite tree $T$, $\ell:E(T) \to \mathbb R_+$ a length function, and $\chi$ a capacity function on $T$.  If $T$ is reduced to a single vertex $\r$, then we have $a_i = 0$ for all $0\leq i<N_{T,\chi}$, by definition, and so the property $\mathscr P_{n+1}$ obviously holds for $T$. Otherwise, consider the following two cases depending on whether $\br(\r)>1$ or $\br(\r)=1$.

\medskip

\noindent (I) \emph{Suppose $\br(\r) >1$}. Let $d=\br(\r)$, and denote by $u_1, \dots, u_d$ the children of $\r$ in $T$.
First note that the first $d$ terms in any sequence $a_0, a_1, \dots$ produced by the weighting process are equal to $0$, since, by the positivity of the values of the length function, the weighting process has to give weight one to all the pending edges at $\r$ before giving weight to any other unweighted edge of $T$.

Consider the subtrees $T_{u_1}, \dots, T_{u_d}$ of  $T$ rooted at $u_1, \dots, u_d$, respectively, with capacity function $\chi_j$ defined as the restriction of $\chi$ to $T_{u_j}$, and set $n_i := \min\{n, N_{T_{u_i}, \chi_j}-1\}$. 
Since $\mathscr P_n$ is verified for all trees, we get for any $i=1, \dots, d$, a well-defined sequence 
$a_{0}^{u_i}, \dots, a_{n_i}^{u_i}$. For each $i$, define the set 
$$A_i := \bigl\{\,a_{0}^{u_i}, a_{1}^{u_i}+\ell_{\r u_i}, \dots, a_{n_i}^{u_i} + n_i\ell_{\r u_i}\bigr\},$$
and let $A$ be the multiset union of the sets $A_i$. Let $ m := n_1+ \dots +n_d +d-1$ and note that $A$ has size $m+1$.
Order the elements of $A$ in an increasing order $b_0 \leq b_1 \leq b_2 \leq \dots \leq b_{m}$. 

\medskip

First note that
\begin{claim}
 We have $m \geq \min\{n+1, N_{T,\chi}-1\}$. 
\end{claim}
\begin{proof} If there is an $1\leq i \leq d$, such that $n\leq N_{T_{u_i}, \chi_i}-1$, we get $m \geq n+ d-1\geq n+1$. Otherwise, we have 
$n_i =N_{T_{u_i},\chi_i}-1$ for all $i\in \{1,\dots,d\}$, and so using that $N_{T,\chi} = \sum_{i=1}^d N_{T_{u_i},\chi_i},$
which comes from the definition, we get $m = (\sum_{i=1}^d N_{T_{u_i},\chi_i}) -1 = N_{T,\chi}-1,$
 and the claim follows. 
\end{proof}

The following claim proves that property $\mathscr P_{n+1}$ holds for any rooted tree $T$ with $\br(\r) \geq 2$. 
\begin{claim}
 The first $n+2$ terms of any sequence $a_0, a_1, \dots$ associated to $(T,\ell, \chi)$ by the weighting process coincide with $b_0, \dots, b_{n+1}$. 
\end{claim}
\begin{proof} For the sake of a contradiction  suppose this is not the case, and consider a weighting process resulting in a sequence $\{a_i\}_{ 0\leq i<N_{T,\chi}}$ such that the claim does not hold, and let $0\leq t\leq n+1$ be the smallest integer with $a_t \neq b_t$. Since $a_0=\dots=a_{d-1} =0$, and $b_0 = \dots =b_{d-1}=0$, we have $t\geq d$. In the ordering 
$b_0\leq b_1 \leq \dots \leq b_t$, each $b_i$ comes from one of the sets $A_1, \dots, A_d$.  Let $0\leq t_1 \leq n_1, \dots, 0\leq t_d\leq n_d$ be integer numbers so that the union of the smallest $t_i+1$ terms in each $A_i$ when reordered in an increasing order gives the sequence of $b_i$s for $0\leq i\leq t$. We note that the sequence $t_1, \dots, t_d$ is not necessarily unique as it might be repetitions among the members of different sets $A_i$. We have 
\[d+ \sum_{i=1}^d t_i =t+1.\]
Consider the weight function $\omega_{t+1}$, and set $s_i:= \omega_{t+1}(\r u_i)-1$.
We have 
\[d+\sum_{i=1}^{d} s_i =t+1,\]
Define $B_i := \bigl\{\,a_{0}^{u_i}, a_{1}^{u_i}+\ell_{\r u_i}, \dots, a_{s_i}^{u_i} + s_i\ell_{\r u_i}\bigr\}.$
Consider the mutiset union $B$ of the sets $B_1, \dots, B_d$.
\begin{claim}\label{claim2.5} The sequence $a_0, \dots, a_t$ coincides with the increasing sequence formed out of the elements of $B$.
\end{claim}
\begin{proof}Follows directly from the definition of the weighting process. Indeed, in any of the cases (1) or (2) in the definition of the weightings and the sequence $\{a_i\}$, if the vertex $x_n$ is  in the subtree $T_{u_i}$, for  $1\leq i \leq d$, then we have 
$a_n = a^{u_i}_{\omega_n(\r u_i) }+ \omega_n(\r u_i)\ell_{\r u_i}$.
\end{proof}

Applying the above claim, since we assumed $a_t \neq b_t$, we infer that the two sequences $(t_1, \dots, t_d)$ and $(s_1, \dots,s_d)$ are different. Let  $1\leq i ,j\leq d$ be the indices which give 
\[a_t = a^{u_{i}}_{s_i} + s_i \ell_{\r u_i}, \quad \textrm{and} \quad b_t = a^{u_{j}}_{t_j} + t_{j} \ell_{\r u_j}.\]

We divide the rest of the proof in two parts depending on whether $a_t>b_t$ or $b_t > a_t$.

Suppose first $a_{t} > b_t$.   Given that $a_{t_i}^{u_i} + t_i \ell_{\r u_i}$ appears among $b_0, \dots, b_t$, we get
$a_{t_i}^{u_i} + t_i \ell_{\r u_i} \leq b_t <a_t = a^{u_i}_{s_i} + s_i \ell_{\r u_i}$. Therefore, the sequence $\{a^{u_i}_j\}$ being increasing, we must have $t_i<s_i$. Since $\sum_{j=1}^d t_j = \sum_{j=1}^d s_j$, there exists an index $h$ so that $s_h< t_h$. It follows that 
\[a_t>b_t \geq a_{t_h}^{u_h}+ t_h \ell_{\r u_h} \geq a_{s_h+1}^{u_h}+ (s_h+1) \ell_{\r  u_h}.\]
This leads to a contradiction. Indeed, the description of the weighting process, the choice of $a_t$, and the fact that $\omega_{t+1}(\r u_h) =s_h$ implies in particular that $a_{s_h+1}^{u_h}+ (s_h+1) \ell_{\r u_h}\geq a_t$. 

\medskip

Suppose now that $b_{t} > a_t$. Then we  have 
\[a_{t_j}^{u_j} + t_j \ell_{\r u_j} > a_t \geq a_{s_j}^{u_j} + s_j \ell_{\r u_j}, \]
which implies $t_j> s_j$. Therefore, there exists an index $1\leq h\leq d$ such that we have $s_h >t_h$. It follows that 
\[a_t \geq a_{s_h} + s_h\ell_{\r u_h} \geq a_{t_h+1} + (t_h+1) \ell_{\r u_h} \geq b_t, \]
which again leads to a contradiction.
\end{proof}

\noindent  (II) \textit{Suppose $\br(\r) = 1$}. Let $P$ be the unique strict path in $T$. If $P$ is infinite, then $N_T =0$, and we are done. Otherwise, let $v$ be the other end of $P$. If $v$ is a leaf, then $T$ is a rooted 
path, and in this case $N_{T,\chi} =\chi(v)$, and we have $a_i = i \ell(P)$ for all $0\leq i< \chi(v)$, and again we are done. So we can suppose that $v$ is branching in $T$.  By case (I), since the root $v$ of $T_v$ is branching, the property $\mathscr P_{n+1}$ is verified for $T_v$. Let $h = \min\{n+1, N_{T_v, \chi}\}$. In particular, the sequence $a_0^v, \dots, a_h^v$ associated to $T_v$ is well-defined. 

We now note that by Formula~\eqref{form:nt}, we have $N_{T,\chi} = N_{T_v,\chi}$. The following claim, which directly follows from the definition of the weighting process, shows that property $\mathscr P_{n+1}$ holds also for any tree $T$ with $\br(\r)=1$. 

\begin{claim}\label{claim2.6}
 For each $0\leq i\leq h$, we have $a_i = a_i^v + i\ell(P)$. \end{claim}

This  finishes the proof of part $(i)$ of Theorem~\ref{thm:main1}. 
\end{proof}
\begin{proof}[Proof of Theorem~\ref{thm:main1}$(ii)$] Denote by $\Gamma$ the metric realization of both $(T_1,\ell_1)$ and $(T_1, \ell_2)$, and let $\chi$ be a capacity function on $T_1$. Note that we have $L(T_1) =L(T_2)$ and  $\mathscr B(T_1) = \mathscr B(T_2)$. It follows that $\chi$ is also a capacity function for $T_2$, and for any vertex $v\in \mathscr B(T_1) = \mathscr B(T_2)$, there is a bijection between the strict paths $P_1$ of the subtree $T_{1,v}$ of $T_1$ rooted at $v$, and the strict paths $P_2$ of the subtree $T_{2,v}$ of $T_2$ rooted at $v$, and moreover, under this bijection, we have $\ell_1(P_1) = \ell_2(P_2)$. 

The choices of vertices producing the factorial sequence in the description of the weighting process for a tree $T$ only depends on the root, branching and leaf vertices,  and the length of strict paths of  subtrees $T_v$ for  branching vertices $v$. It follows that any factorial-determining sequence  $x_n$ in $(T_1,\ell_1,\chi)$ is also factorial-determining sequence in $(T_2, \ell_2,\chi)$, from which the theorem follows. 
\end{proof}
We state the following useful recursive \emph{min-max formula} for the factorials obtained in the above proof.
\begin{thm} Let $(T, \ell,\chi)$ be a triple of a locally finite tree $T$ rooted at $\r$, a length function $\ell$  and a capacity function $\chi$ on $T$, respectively. Let $d=\br(\r)$ and denote by $u_1, \dots, u_d$ all the children of $\r$. For each $j=1,\dots, d$, let $N_j = N_{T_{u_j},  \chi_j}$, for the restriction $\chi_j$ $\chi$ to $T_{u_j}$, and denote by $\ell_j$ the restriction of $\ell$ to $T_{u_j}$. Then, for all integers $0\leq n<N_{T,  \chi}$, we have
$$n!_{(T,\ell,\chi)}  = \min_{\substack{(n_1, \dots, n_d) \in \mathbb N^d\\  \textrm{for all $j$, }\,0\leq n_j < N_j\\
n_1+ \dots +n_d = n+1}} \max\,\, \Bigl\{\,(n_j-1)!_{(T_{u_j}, \,\ell_j,\chi_j)} + (n_j -1) \ell_{\r u_j}\,\Bigr\}_{j=1}^d.$$
\end{thm}

\subsection{Proofs of Theorem~\ref{thm:main1-intro} and Theorem~\ref{thm:main1bis-intro}}
We first prove the equivalence of the definition of the factorial sequence given in Section~\ref{sec:locfin-intro} with the one given in this section, thus proving Theorem~\ref{thm:main1bis-intro}. 
\begin{thm} Let $T$ be a locally finite rooted tree, $\ell$ a length function and $\chi$ a capacity function on $T$. Let $\Gamma$ be the metric tree with a model $(T, \ell)$. Let $\alpha_0, \alpha_1,\dots$ be the sequence of numbers associated to $(\Gamma,\chi)$ as in~\eqref{def:fac-intro}. We have for all $n\in \mathbb N$, $(n!)_{(T,\ell,\chi)} = \alpha_n$. 
\end{thm}
\begin{proof} Let $\rho_0, \rho_1, \dots$ be a sequence of elements of the extended boundary $\widetilde \partial \Gamma$ producing $\alpha_0,\alpha_1, \dots$, as in~\eqref{def:fac-intro} in the introduction, i.e., 
\[\alpha_n =\sum_{i=0}^{n-1} \langle \rho_n, \rho_i\rangle.\]
Proceeding by induction on $n$, we show how the sequence $\{\rho_i\}$  define a weighting sequence $\{\omega_i\}$, tree factorials $T_{\omega_i}$, and vertices $\{x_i\}$,  so that we have $a_n =\alpha_n$. The weighting sequence $\{\omega_i\}$ is defined in such a way that for each $n$, the property $\mathscr Z_n$ is verified:

\begin{itemize}
\item[$(\mathscr Z_n)$] $\bullet $ for each edge $uv $ in $T_{\omega_n}$, $\omega_n(uv)$ is the number of elements $\rho$ in the sequence $\rho_0, \dots, \rho_{n-1}$ which belong to the extended boundary $\widetilde \partial T_{v}$, and 
\end{itemize}
\begin{itemize}
\item[]   $\bullet$ For each leaf $v$  of $T_{\omega_n}$, we have $\omega_n(\cev vv)=1$.
 \end{itemize}For $n=0$, the element $\rho_0$ of the extended boundary starts with a strict $P_0$ of $T$, with  one end point $\r$. Let $\omega_1$ be the weighting associated to the choice of $P_0$ in the weighting process, and note that  $(\mathscr Z_1)$  obviously holds. Proceeding recursively, suppose that for $n\geq 1$, $\omega_0, \dots, \omega_n$  and $x_1, \dots, x_{n-1}$ are defined, so that $(\mathscr Z_n)$ holds. Consider $\rho_n \in \widetilde \partial \Gamma$, and define $x_{n}$ as the last vertex of $T_{\omega_n}$ on the path $\rho_n$. By the definition of the sequence $\rho_i$, the vertex  $x_n$ is an unsaturated vertex of $T_{\omega_n}$ and so belongs to $U_n$. In addition, by Property $(\mathscr Z_n)$, we have 
 $$\ell_{\omega_n} = \sum_{e \in [\r, x_n]} \omega_n(e) \ell_e = \sum_{j=0}^{n-1}\langle \rho_n, \rho_j\rangle,$$
 which shows that $x_n$ is a vertex in $U_n$ which minimizes the $\ell_{\omega_{n}}$-distance to the root $\r$. 
 
 Two cases can happen. If $x_n$ is not clear, let $x_ny$ be the pending edge at $x_n$ which belongs to $\rho_n$, and let $P_y$ be the corresponding strict path in $T_{x_n}$. Define $\omega_{n+1}$ as Case (1) in the definition of the weighting process. One easily verifies that $(\mathscr Z_{n+1})$ holds.  
 
 Otherwise, $x_n$ is a clear vertex. If $x_n$ is a leaf of $T$, then define $\omega_{n+1}$ as in Case (2.1) in the the definition of the weighting process. Otherwise, $x_n$ is branching, and is a leaf of $T_{\omega_n}$. By Property $(\mathscr Z_n)$, we have $\omega_{n}(\cev x_{n} x_n)=1$ and so there exists a unique $0\leq j\leq n-1$ so that $\rho_j$ is in $\widetilde \partial T_{x_n}$.  Let $x_nz$ and $x_n w$ be the two pending edges at $x_n$ which belong to $\rho_j$ and $\rho_n$, respectively, and define $\omega_{n+1}$ as in Case (2.2) in the definition of the weighting process. One easily verifies that in both cases, Property $(\mathscr Z_{n+1})$ is verified. 
 \end{proof}

To prove Theorem~\ref{thm:main1-intro}, it will be now enough to prove Proposition~\ref{prop:reduction}.

\begin{proof}[Proof of Proposition~\ref{prop:reduction}]
Let $(T_0, \ell_0,\chi_0)$ be associated to $(T, \ell)$ as in Section~\ref{sec:reduction}. The proof  is based on the observation that for any vertex $v$ of valence infinity in $T$, and for any sequence $\rho_0, \rho_1, \dots$ defining the sequence $a_0, a_1, \dots$ for the pair $(T, \ell)$ as in the introduction,  any two different elements $\rho_i$ and $\rho_j$, for $i\neq j$ which belong to $\partial \widetilde T_v$, must contain two different pending edges $v w$ and $vz$ at $v$.   The sequence $\rho_0, \rho_1, \dots$ gives a sequence $\tau_{0}, \tau_{1}, \dots$, where $\tau_{i}$ is defined by intersecting $\rho_i$ with $T_0$. Using the observation, one  verifies that for any $n$, 
\[a_n = \sum_{j=0}^{n-1} \langle \rho_n, \rho_j\rangle = \sum_{j=0}^{n-1}\langle \tau_{n}, \tau_{j}\rangle. \]
The Proposition now follows from Theorem~\ref{thm:main1bis-intro}, since $\tau_{0}, \tau_{1},\dots$ determines the factorials of $(T_0, \ell_0, \chi_0)$.
\end{proof}

In the rest of this section, we prove some basic fundamental results which will be used in the upcoming sections. 
\subsection{The exhaustiveness of the weighting process}
The following proposition shows the weighting process eventually gives weight to any edge of the tree.  
\begin{prop}\label{prop:cover}
 Let $(T,\ell)$ be a pair consisting of a locally finite tree with a length function $\ell$, and let $\chi$ be a capacity function on $T$. Any sequence of weighting $\{\omega_n\}$ producing the logarithmic
 factorials of $(T,\ell,\chi)$ eventually gives a weight to any edge of the tree $T$. In other words, the union of the trees in any sequence of factorial trees  is the whole tree $T$.
 \end{prop}

\begin{proof} 
Let $N = N_{T, \chi}$. 
 Let $T_0 = \bigcup_{i=0}^{N_T} T_{\omega_i}$, and for the sake of a contradiction, suppose $E(T)\setminus E(T_0) \neq \emptyset$. Since $T_0$ is a tree, there exists a vertex $v$ in $V(T) \setminus V(T_0)$ adjacent to a vertex $u\in T_0$, so we have $\cev v=u$. 

As already observed before, the first 
 $\br(\r)$ terms of the factorial sequence are all $0$, and the weighting consists in giving weights to the edges of the strict paths of $T$, in particular to all the pending edges of $T$ at $\r$.

 Note that since $u$ is branching, it  remains unsaturated in all the trees $T_{\omega_i}$ which contain $u$. This is impossible if the factorial  sequence is finite, so we can suppose that $N = \infty$.

 \medskip
 
 To simplify the presentation, we will use the usual $O$ notation in what follows: for a sequence of non-negative number $\{f_n\}_{n\in \mathbb N}$, we write $f_n =O(1)$ if there exists a constant $C>0$ such that for all $n\in \mathbb N$, we have $f_n \leq C$. 
  \medskip
 
 Denote by $P=\r u_1 u_2 \dots u_k$ the oriented path $[\r,u]$ from $\r$ to $u_k = u$ in $T$. 
 \begin{claim}\label{claim:boundedness}
 For any $1\leq j\leq k$, and any pending edge $u_jx$ at $u_j$ in $T_0$, the sequence $\Bigl\{\omega_n(u_jx)\, \bigl|\, n\in \mathbb N \textrm{ with } u_jx\in T_{\omega_n} \Bigr\}$ verifies $\omega_n(u_jx) = O(1)$.
 The same statement holds for the sequence $\Bigl\{\omega_n(\r u_1)\, \bigl|\, n\in \mathbb N \textrm{ with } \r u_1\in T_{\omega_n} \Bigr\}$ 
 \end{claim}
 
 We prove the first part of the claim by a reverse induction on $j$.

   Let $x_0, x_1, \dots$ be the factorial-determining sequence corresponding to the edge weighing sequence $\omega_0, \omega_1, \dots$.   Consider first the vertex $u_k = u$, and let $u_kx$ be a pending edge at $u_k$ in $T_0$. We claim that $\omega_n(u_kx)\leq 1$ for all $n$ with $u_k x\in T_{\omega_n}$.  Suppose this is not the case, and consider the integer $m$ such that $x_m$ is chosen in the subtree $T_{x}$, and $\omega_m(u_kx)=1$ and $\omega_{m+1}(u_kx)=2$. We have $\ell_{\omega_m}\bigl([\r,x_m]\bigr) > \ell_{\omega_m}\bigl([\r,u_k]\bigr)$, which shows that the choice of $v$ instead of $x_n$ gives a strictly smaller value for $a_n$ (in the weighting process). This contradiction proves the claim for $j=k$. 

Suppose now that the claim holds for all integers $i$ with $1\leq j<i \leq k$. We prove that the claim holds for integer $j$. Let $u_jx$ be a pending edge at $u_j$ in $T_0$, and suppose that $\omega_n(u_jx)$ tends to infinity, as $n$ tends to infinity. By the hypothesis of the induction, we have $\omega_n(u_iu_{i+1})  = O(1)$ for all $j+1\leq i \leq k$ and for all large enough integers $n$. In addition, since for any large $n$, we have $\omega_n(u_ju_{j+1}) =\sum_{u_{j+1}x \in T_0} \omega_n(u_{j+1}x)$, we infer that $\omega_n(u_j u_{j+1}) = O(1)$. This in particular implies that 
$\ell_{\omega_n}([u_j, u_k]) \leq C$, for some constant $C>0$ and all large enough integers $n$.  Let $m$ be an integer  such that 
$\omega_{m}(u_jx) > C/\ell(u_jx)$ and $\omega_{m+1}(u_jx) = \omega_{m}(u_jx)+1$, which exists by the assumption that $\omega_n(u_jx)$ tends to infinity.  The vertex $x_{m}$ lies in the subtree $T_x$. 
We have 
\begin{align*}
\ell_{\omega_n}\bigl([\r,x_m]\bigr) \geq \ell_{\omega_m}\bigl([\r,x]\bigr) &= \ell_{\omega_m}\bigl([\r,u_j]\bigr)  + \ell_{\omega_m}\bigl([u_j,x]\bigr)  = \ell_{\omega_m}\bigl([\r,u_j]\bigr)   + \omega_m(u_jx) \ell(u_jx) \\
&> \ell_{\omega_m}\bigl([\r,u_j]\bigr)   + C \geq \ell_{\omega_m}\bigl([\r,u_j]\bigr) + \ell_{\omega_m}\bigl([u_j,u_k]\bigr)  \\
&= \ell_{\omega_m}\bigl([\r,u_k]\bigr), 
\end{align*}
which is a contradiction with the choice of $x_m$. This proves the claim for all the pending edges at $u_j$, $j=1,\dots, k$. 
The boundedness of the sequence $\omega_n(\r u_1)$ follows now from the fact that $\omega_n(\r u_1) = \sum_{u_1x \in E(T_0)} \omega_n(u_1x)$ for all large enough integers $n$.

\medskip

  To finish the proof of the proposition, note that since $N_{T, \chi}=\infty,$ we have $\sum_{\r x \in E(T_0)} \omega_n(\r x) \to \infty$, which shows that $n!_{(T,\ell, \chi)} \to \infty$ as $n$ tends to infinity. 
   On the other hand, by the claim we just proved, we have $\ell_{\omega_n}\bigl([\r, u_k]\bigr) \leq C$ for some constant $C>0$ and all large enough integers $n\in\mathbb N$. But this is impossible since at some stage $n$, the weighting process will have the better choice of $v$ instead  of $x_n$.   
\end{proof}

\subsection{Super-additivity of the factorial sequence}
We now prove the following useful proposition.
\begin{prop}\label{prop:supad}
 For any pair triple $(T, \ell,\chi)$  consisting of a locally finite tree $T$, a length function $\ell$ and a capacity $\chi$ on $T$, and for all non-negative integers $n\geq m$, we have 
 $$n!_{(T,\ell,\chi)} \geq (n-m)!_{(T,\ell,\chi)} + m!_{(T,\ell,\chi)}.$$
\end{prop}

 Applying Fekete's Lemma, the proposition implies that
\begin{cor}\label{cor:limit}
 For any triple $(T, \ell,\chi)$ with $N_{T, \chi} =\infty$, the limit of the sequence $\frac 1n (n!)_{(T,\ell,\chi)}$ exists and belongs to the interval $(0,+\infty]$. 
\end{cor}
We will later describe the value of the limit. 
\begin{proof}[Proof of Proposition~\ref{prop:supad}]
We prove by induction the following property $\mathscr Q_M$, for non-negative integers $M$.

\medskip

\noindent$(\mathscr Q_M)\,\,\,\,$ \emph{For all triples $(T, \ell,\chi)$ consisting of  a rooted tree $T$, length function $\ell$ and capacity function $\chi$ on $T$, and for all non-negative integers $n\leq M$ and $0\leq m\leq n$, we have  
$$n!_{(T,\ell,\chi)} \geq (n-m)!_{(T,\ell,\chi)} + m!_{(T,\ell,\chi)}.$$}

The result obviously holds for $M=0$. So suppose $\mathscr Q_n$ holds for all $n<M$. We prove $\mathscr Q_M$.

Using Claim~\ref{claim2.6}, we can reduce to  the case where the root is branching, i.e., $d:=\br(\r)\geq 2$. Denote by $u_1, \dots, u_d$ the children of $\r$, and denote by $\ell_j$ and $\chi_j$ the restriction of $\chi$ to $T_{u_j}$, respectively. Let $N_j := N_{T_{u_j},\chi_j}$. For $j=1, \dots, d$, denote by 
$S_{j}:=\bigl\{a_i^{j}\bigr\}_{0\leq i<N_j}$ the set of factorials of $(T_{u_j}, \ell_j, \chi_j)$ with $a_i^{j} := i!_{(T_{u_j},\ell_j, \chi_j)}$. Define $A_j := \bigl\{a_i^j + i  \ell_{\r u_j} \bigr\}_{0\leq i< N_{j}},$ as in the proof of Theorem~\ref{thm:main1}.  

Let $0\leq  m \leq n\leq M$ be two integers. We show the inequality $n!_{(T,\ell,\chi)} \geq (n-m)!_{(T, \ell,\chi)} + m!_{(T, \ell,\chi)}$.  We can suppose that $n=M$, as otherwise, the inequality follows from   the validity of $\mathscr Q_{M-1}$.

By Claim~\ref{claim2.5}, $M!_{(T,\ell,\chi)}$ is the $(M+1)$-st term in the multiset union $A$ of the sets $A_1, \dots, A_d$ when the terms are put in an increasing order. We label each element of the multiset union $A$ with the index $j$ of the set $A_j$ where it comes from, and fix an increasing order on the elements of $A$. In this way we can define positive integer numbers $M_j$, for $j=1,\dots, d$, as the number of terms labeled with $j$ among the first $M+1$ terms. In other words, for each $1\leq j\leq d$, 
$a_i^j+i \ell_{\r u_j}$ for $i=0, \dots, M_j-1$ are among the $M+1$ first terms of $A$. In particular, we have $\sum_{j=1}^d M_j = M+1,$ \textrm{and} 
\begin{equation}\label{eq:sa1}
\textrm{for all $1\leq j\leq d$,}\,\, \,\, M!_{(T,\ell, \chi)} \geq a_{M_j-1}^j+(M_j-1) \ell_{\r u_j},\,\, \textrm{with equality for at least one $j$}.
\end{equation}

Similarly, the $(m+1)$-st term in $A$ is equal to $m!_{(T,\ell, \chi)}$, and we define $m_j$, for each $j=1, \dots, d$, as the number of terms of $A_j$ which appear in the first $m+1$ terms of  $A$. We have $\sum_{j=1}^d m_j = m+1,$ and
\[ \textrm{for all $1\leq j\leq d$,} \,\,\,\, m!_{(T,\ell, \chi)} \geq a_{m_j-1}^j+(m_j-1) \ell_{\r u_j} \,\, \textrm{with equality for at least one $j$}.\]
In addition, since all the terms of the form $a_{m_j}^j+m_j \ell_{\r u_j}$, for $j=1, \dots, d$, appear after the $(m+1)$-st term of the sequence of $A$, it follows that
\begin{equation}\label{eq:sa2}
\forall \,\,\, 1\leq j \leq d, \qquad a_{m_j}^j+m_j \ell_{\r u_j} \geq m!_{(T,\ell, \chi)}.
\end{equation}
Suppose without loss of generality that  $m!_{(T,\ell, \chi)} = a_{m_1-1}^1+(m_1-1) \ell_{\r u_1}$. Note that we have $M_j \geq m_j$ for all $j$.

\medskip

Combining Inequalities~\eqref{eq:sa1} and~\eqref{eq:sa2}, we get
\begin{align*}
M!_{(T, \ell, \chi)} - m!_{(T, \ell, \chi)} &\geq \Bigl(a^1_{M_1-1}+(M_1-1) \ell_{\r u_1}\Bigr) -  \Bigl(a_{m_1-1}^1+(m_1-1)\ell_{\r u_1}\Bigr) \\
&= \Bigl(a^1_{M_1-1} - a_{m_1-1}^1 \Bigr) + (M_1 -m_1) \ell_{\r u_1},
\end{align*}
and for all values of $j\geq 2$ with $M_j> m_j$, we have
\begin{align*}
M!_{(T, \ell, \chi)} -m!_{(T,\ell, \chi)} &\geq \Bigl(a_{M_j-1}^j+(M_j-1) \ell_{\r u_j}\Bigr) - \Bigl(a_{m_j}^j+m_j\ell_{\r u_j} \Bigr)\\
&= \Bigl(a_{M_j-1}^j  - a_{m_j}^j \Bigr) +(M_j- m_j - 1) \ell_{\r u_j}.
\end{align*}

Since $M_j \leq M-1$ for all $j$, by property $\mathscr Q_{M-1}$ applied to the subtrees $T_{u_j}$, we get for all $j\geq 2$ with $M_j > m_j$,
\[\Bigl(a_{M_j-1}^j  - a_{m_j}^j \Bigr) +(M_j- m_j - 1) \ell_{\r u_j} \geq a^j_{M_j-1-m_j} + (M_j- m_j - 1) \ell_{\r u_j}.\]
Moreover, for $j=1$, we have 
\[\Bigl(a^1_{M_1-1} - a_{m_1-1}^1 \Bigr) + (M_1 -m_1) \ell_{\r u_1} \geq a_{M_1-m_1}^1+(M_1-m_1) \ell_{\r u_1}.\]
We infer that all the terms of the form $a_i^j+i \ell_{\r u_j}$ for $j \geq 2$ and $0\leq i\leq M_j-m_j-1$, and all the terms $a_i^1+i \ell_{\r u_1}$ for 
$0\leq i\leq M_1-m_1$ are bounded from above by $M!_{(T, \ell, \chi)} -m!_{(T,\ell, \chi)}$. This shows that 
in the multiset union $A=\bigcup_{j=1}^d A_j$, there are at least 
$$M_1-m_1 +1 + \sum_{\substack{2\leq j\leq d\\
\textrm{such that}\,\,M_j>m_j}} (\,M_j-m_j\,) = 1+ \Bigl(\,\sum_{j=1}^d M_j-m_j\,\Bigr) = M-m+1
$$
terms bounded from above by $M!_{(T, \ell, \chi)} -m!_{(T,\ell, \chi)}$. Since the $(M-m+1)$-st term in the sequence of elements of $A$ is $(M-m)!_{(T,\ell, \chi)}$, we finally get the required inequality
\[M!_{(T,\ell, \chi)} -m!_{(T,\ell, \chi)}\geq (M-m)!_{(T,\ell, \chi)}.\]
\end{proof}

\section{How much information factorial sequence gives about the tree?} \label{sec:hminf} 

 It is natural to ask how much information about the tree $T$ is captured by the factorial sequence and, in particular, whether the factorial sequence associated to $T$ characterizes $T$ uniquely? The question is intimately related to the question of characterizing the sequences of integers which can be realized as factorials associated to a tree. In this section we discuss these questions. Let us make the following definition.

\begin{defn}[Factorial realizability] \rm Let $N \in \mathbb N \cup\{\infty\}$. A sequence  $S = \{a_i\}_{0\leq i<N}$ of increasing non-negative real numbers is \emph{factorial realizable} if there exists a locally finite rooted tree $T$, a length function $\ell$ and a capacity function $\chi$ on $T$ such that for each non-negative integer $0\leq n< N$, we have $a_n = n!_{(T, \ell ,\chi)}$.  
\end{defn}

\subsection{Realizability of sufficiently biased sequences}

 Consider an infinite increasing sequence of positive numbers $S$. Let $d \in \mathbb N$ be a natural number. 
 We can rewrite the elements of $S$ in the form (in an increasing order) 
$$a_{0,1}, \dots, a_{0,d}, a_{1,1}, \dots, a_{1,d}, a_{2, 1}, \dots, a_{2,d^2}, \dots, a_{n,1}, \dots, a_{n,d^n}, \dots $$ 
consisting for each $n$ of $d^n$ reals $a_{n,1}\leq \dots \leq a_{n,d^n}$. 
\begin{defn} \rm An  infinite increasing sequence $S$ is called \emph{$d$-sufficiently biased} if it satisfies:
$$\textrm{for each $n \geq 0$}, \qquad a_{n+1, 1} > 2d^{n+1} \sum_{i=0}^n a_{i,2^i}.$$ 
\end{defn}

Let $d\geq 2$ be an integer. Let $\mathscr T_d$ be the rooted $d$-regular tree, where every node has branching equal to $d$, and for each integer $n\geq 0$, choose an arbitrary total order $\leq_n$ on the vertices of the $\mathscr T_d$ at generation $n$.  Let $\leq$ be the total order on the vertices of $\mathscr T_d$ induced by  total orders $\leq_n$, and by declaring $u<v$ for two vertices $u,v$ of $\mathscr T_d$ provided that the vertex $u$ has generation strictly smaller than that of $v$. 

\begin{defn}\rm Given a collection of total orders $\{\leq_n\}_{n=0}^\infty$ inducing a total order $\leq$ on the nodes of $\mathscr T_d$ as above, and given a length function $\ell:E(\mathscr T_d) \rightarrow \mathbb N$, we say that  $\ell$ and $\leq$ are \emph{coherent} if the following two properties hold in the 
construction of the factorials associated to $(T,\ell)$: 
\begin{itemize}
 \item[(1)] for each non-negative integer $n$, all the vertices in generation $n+1$ are clear as far as there exists an unsaturated vertex in generation $n$;
 \item[(2)] the  order of weighting unweighted pending edges at vertices of generation $n-1$  of $\mathscr T_d$  in the weighting process coincides with the total order $\leq_n$ on generation $n$.
\end{itemize}
\end{defn}

Note in particular that (1) implies that for any $n$, a vertex $x_k$ of $\mathscr T_d$ which gives $k!_{(T,\ell)}$ in the construction of the factorial sequence lies in generation $n$ provided that  $d^{n}\leq k \leq d^{n+1}-1$. 

We have the following theorem. 
\begin{thm}\label{thm:realizability} Let $d\geq 2$ be an integer. For any $d$-sufficiently biased sequence $S$ as above with $a_{0,1}= a_{0,2} =\dots = a_{0,d} = 0$, and any collection $\Bigl\{\leq_n\Bigr\}_{n\in \mathbb N}$ of total orders $\leq_n$ on the $n$-th generation of the $d$-regular tree $\mathscr T_d$, there is a length function $\ell : E(\mathscr T_d) \rightarrow \mathbb R_+$ 
such that 
\begin{itemize}
 \item $\ell$ and $\mathcal O$ are coherent, and
 \item the factorial sequence associated to the pair $(\mathscr T_d, \ell)$ coincides with $S$.
\end{itemize}
 \end{thm}

 \begin{proof} We describe how to construct the length function by induction.  
 
 For each $n\in \mathbb N$, denote by $T_n$ the subtree of $\mathscr T_d$ consisting of all the vertices at generation not exceeding $n$. Denote by $u_{n,1} <_n \dots <_n u_{n,d^n}$ 
 all the vertices of generation $n$ in an increasing order with respect to the total order $\leq_n$. For each $1\leq i \leq d^n$, denote by $e_{n,i}$ the unique edge of $T$ joining a vertex of generation $n-1$ to $u_{n,i}$. 
 So, for example, $u_{1,1} <_1 u_{1,2} < \dots < u_{1,d}$ are the $d$ vertices of $\mathscr T_d$ adjacent to the root $\r$, and we have $e_{1,1} = \r u_{1,1}, \dots, e_{1,d} = \r u_{1,d}$. 
  
  \medskip
  
 First consider $n=1$. Define for each $i=1,\dots, d$, the length of $e_{1,i}$ by $\ell(e_{1,i}) := a_{1,i}$. Proceeding inductively, suppose that the length of all the edges in the tree $T_n$ have been already assigned, and that the lengths of edges of $T_n$ verify the following property $(\mathscr L_n)$
\[(\mathscr L_n)\qquad \qquad  \forall\,\,\,  1\leq j \leq n \textrm{ and }\,\, 1\leq k \leq 2^j\,,\,\qquad {\frac{a_{j,1}}2} \leq \ell(e_{j,k}) \leq a_{j,k}\,.\]

We now explain how to define $\ell(e_{n+1, i})$ for all $1\leq i \leq d^{n+1}$, ensuring the property $(\mathscr L_{n+1})$ as well. 

Fix an $1\leq i\leq d^{n+1}$, and let $P_i = v_0 v_1 v_2 \dots v_n v_{n+1}$ be the unique path from the root $\r =v_0$ to $v_{n+1} = u_{n+1, i}$. Note that $e_{n+1,i} = v_{n}v_{n+1}$.
For each $1\leq j \leq n$, 
let $f_{i,j}$ be the number of descendants of $v_j$ among the vertices $u_{n+1, 1}, \dots, u_{n+1, i-1}$. Obviously, we have $0\leq f_{i,j} \leq d^{n+1-j}$ for any $1\leq j \leq n $, and, we have $f_{1,1} = \dots, = f_{1,n}=0$.
Define $\ell(e_{n+1,i})$ by the recursive equation
\[\ell(e_{n+1, i} )+ \sum_{j=1}^n \bigl(d^{n+1-j}+f_{i,j}\bigr)\ell(v_{j-1}v_j) = a_{n+1,i}.\]
 By property $(\mathscr L_n)$, since $S$ is $d$-sufficiently biased, we have 
\[\sum_{j=1}^n \bigl(d^{n+1-j}+f_{i,j}\bigr)\ell(v_{j-1}v_j) \leq d^{n+1} \sum_{j=1}^n a_{j, d^j} < \frac12 a_{n+1,1} \leq a_{n+1,i},\]
which ensures that  
$$ \frac{a_{n+1,1}}2\leq \ell(e_{n+1, j}) \leq  a_{n+1,i}.$$

To prove that $\ell$ and $\mathcal O$ are coherent, and that the factorial sequence associated to $(T,\ell)$ coincides with $S$ one can proceed by induction. The details are straightforward and are left to the reader.     
 \end{proof}

\subsection{Module of definition of the length function} 
Let $S$ be an increasing sequence of positive numbers. Denote by $\mathbb Z\langle S \rangle$ the $\mathbb Z$-submodule of $\mathbb R$ generated by the elements of $S$. We have the following proposition.
\begin{prop}\label{prop:mdef} Let $S$ be an increasing sequence of positive numbers realizable by a pair $(T, \ell)$. For any edge $e\in E(T)$, we have $\ell(e) \in \mathbb Z \langle S \rangle$.
\end{prop}
\begin{proof} Let $x_1, x_2, \dots$ be a factorial-determining sequence of vertices for $(T, \ell)$. By exhaustiveness of the weighting process, for each vertex $v$ of the tree, there exists an integer $n$ so that $x_n=v$, that we suppose in addition to be the smallest such $n$.  Write the path $[\r,v]$ in $T$ as $v_0v_1\dots v_k$ with $v_0=\r$ and $v_k=v$. Since $\omega_n(v_{k-1}v_k)=1$, we have 
$n!_{(T,\ell)} = \ell_{\omega_n}([\r,v]) = \ell(v_{k-1}v_k) + \ell_{\omega_n}([\r, v_{k-1}])$. 
Which gives
\[ \ell(v_{k-1}v_k)  =  n!_{(T,\ell)} -  \ell_{\omega_n}([\r, v_{k-1}]).\]
Using this observation, a straightforward induction gives  $\ell(e) \in \mathbb Z\langle S\rangle$ for any $e \in E(T)$.
\end{proof}
\subsection{Two non-isomorphic trees with the same factorial sequence}
The following direct corollary of Theorem~\ref{thm:realizability} and Proposition~\ref{prop:mdef} shows that the factorial sequence cannot determine the tree in general.
\begin{prop}\label{prop:neg}
 There are non-isomorphic trees $T_1$ and $T_2$ with the same factorial sequence.
\end{prop}

\begin{proof}
Consider a $d$-regular tree with $d\geq 2$. For any $d$-sufficiently biased sequence $S$ of integers, and  any collection of total orders on the $n$-th generation of the $d$-regular tree $\mathscr T_d$, for $n\in \mathbb N$, there is a length function $\ell$ such that factorial sequence of $(T, \ell)$ coincides with $S$. Changing the total orders $\leq_n$ associates another length function $\ell'$ with the same factorials. Note that by Proposition~\ref{prop:mdef}, the length function $\ell$ and $\ell'$ are integer valued. 

However, one can easily construct 
a $d$-sufficiently biased sequence such that the two metric trees $\Gamma$ and $\Gamma'$ associated to $(T, \ell)$ and  $(T, \ell')$, respectively, become non-isomorphic. Each of $\Gamma$ and $\Gamma'$ has a model with a standard metric (i.e., with length function equal to one on edges). This results in two non-isomorphic trees $T_1$ and $T_2$ with the same factorials.\end{proof}

\begin{remark}\rm A variant of the  construction of Theorem~\ref{thm:realizability} leads to the following stronger statement. Let $T_1, T_2$ be any pair of infinite rooted  locally finite trees, with roots $\r_1$ and $\r_2$, respectively, so that $\br(\r_1) = \br(\r_2)$. Suppose that all the vertices of $T_1$ and $T_2$ are branching.
There exist length functions $\ell_1: E(T_1) \rightarrow \mathbb N$ and 
$\ell_2: E(T_2) \rightarrow \mathbb N$ so that the factorial sequences associated to $(T_1, \ell_1)$ and $(T_2, \ell_2)$ coincide. The proof goes as follows. 
One shows that for any infinite locally finite tree $T$ in which every vertex is branching, and for any fixed total order on the vertices of generation $n$,  and 
any appropriately biased sequence $S$ of integers with respect to $T$ (a modification of the definition for $d$-regular trees which takes into account the structure of $T$), there exists a length function $\ell: E(T) \rightarrow \mathbb N$ such that the factorial sequence associated to $(T, \ell)$ coincides with $S$. For any sequence which is biased for both $T_1, T_2$, this leads to the statement. 
Since we do not have any utility for this stronger version, we omit the detailed proof.
\end{remark}

\subsection{The case of two trees one included in the other}
In this section, we prove that, not surprising, if one of two trees is included in the other one, and the two factorial sequences are the same, then the two trees are the same.

\begin{prop} Let $(T,\ell,\chi)$ and $(T',\ell',\chi')$ be two triples consisting of rooted locally finite trees with length and capacity functions, so that  
  $T \subseteq T'$, and the restriction of $\ell'$ (resp. $\chi'$) to $T$ coincides with $\ell$ (resp. $\chi$), and so that $(T, \ell,\chi)$ and $(T', \ell',\chi')$ have the same factorial sequence. Then the inclusion induces an isomorphism $T = T'$ (and so $\ell = \ell'$ and $\chi=\chi'$).
\end{prop}
\begin{proof} For the sake of a contradiction, assume $T\subsetneq T'$ and $n!_{(T, \ell,\chi)} = n!_{(T',\ell',\chi')}$ for all $n \in \mathbb N$.  There exists a vertex $v$  of $T'\setminus T$ such that  $\cev v$ belongs to $T$.  
Consider  a sequence of weighting $\omega_n$ for $T$ 
which provides the factorial sequence for $(T,\ell)$, as in Section~\ref{sec:def}. By the equality of the factorial sequences of $(T, \ell,\chi)$ and $(T',\ell',\chi')$, and since 
$T \subset T'$ and the length and capacity functions coincide on $T'$, the same weighting sequence provides the factorial sequence in $T'$. 
This is however impossible since by Proposition~\ref{prop:cover} any weighting sequence eventually gives a 
weight to any edge of the tree $T'$, while the edge $\cev vv$ in $T'$ clearly remains weightless in the sequence
$\omega_n$. 
\end{proof}

Let $d\geq 2$ be an integer. Consider the family $\mathcal T_d$ of all locally finite trees $T$ containing only branching vertices with $2\leq \br(v)\leq d$ for any vertex $v$. 

\begin{question} \emph{Prove or disprove:} for any pair of trees $T_1, T_2\in \mathcal T_d$ with the same factorial sequence, the two trees $T_1$ and $T_2$ are isomorphic.
\end{question}

\section{Growth of the factorial sequence: transience and equidistribution}\label{sec:growth}
The examples given in the previous section of trees with 
any sufficiently biased  sequence of reals as the factorial sequence 
show that the factorials might have any atypical behavior. In this section, we prove Theorems~\ref{equivalence-intro} and~\ref{thm:main2-intro}, which show however some asymptotic regularity behavior  when $n$ tends to infinity.

We first recall some basic definitions and results on random walks and flows on locally finite infinite 
trees. 

Let $(T, \ell)$ be a pair consisting of a locally finite rooted tree and a length function $\ell: E \rightarrow \mathbb R_+$. We assume as before that 
the edges of $T$ are oriented away from the root; $E(T)$ denotes the oriented edges of $T$ with this orientation.  A \emph{flow} $\sigma$ on $T$ consists of an application 
$\sigma: E(T) \rightarrow \mathbb R_{\geq 0}$ such that for any vertex $v\neq \r$ of $T$, we have 
$$\sigma(\cev{v}v) = \sum_{vu \in E(T)} \sigma(vu).$$

The \emph{total amount} of a flow $\sigma$ is by definition the quantity $\sum_{\r v\in E(T)} \sigma(\r v)$, and if this sum is equal to one, then $\sigma$ is called a \emph{unit flow}.
We denote by $\mathscr{F}_u(T)$ the set of all unit flows on $T$. Note that $\mathscr{F}_u(T)$ is non-empty if and only if $T$ is infinite.

Denote, as before, by $\partial T$ the boundary of $T$, which consists of infinite (oriented) paths $s$ which start from the 
root $\r$. The boundary  comes with a natural topology induced by the tree structure. Recall that a basis of non-empty 
open sets in this topology are the sets $B_v$ in bijection with the vertices of $v$, where for a vertex $v\in V(T)$, the set $B_v$ contains all the elements $s \in \partial T$ which contain $v$, i.e., 
\[B_v :=\Bigl\{\,s \in \partial T \,|\, v\in s\,\Bigr\}.\]
Any unit flow $\theta \in \mathscr F_u(T)$ defines naturally a measure of total mass one on $\partial T$ by 
\[\forall v\in T,\qquad \theta(B_v):=\theta(\cev v\, v).\]
This association of measures to unit follows induces a bijection from $\mathscr F_u(T)$ to the set of Borel measures of total mass one on $\partial T$.

\medskip

 Consider the $L^2$-space of real-valued functions on the edges of the tree
 $$L_\ell^2(E) :=\Bigl\{\,\theta: E \rightarrow \mathbb R \,\,\bigl|\,\, \sum_{e\in E}\,  \ell(e) \theta(e)^2<\infty\,\Bigr\},$$
 with the scalar product $\langle\cdot\,,\cdot\rangle$ given by 
 \[\forall\,\,\theta_1, \theta_2 \in L^2(E)\qquad \quad \langle \theta_1, \theta_2\rangle: = \sum_{e}\ell(e)\theta_1(e)\theta_2(e).\]
We may refer to the norm squared of an element $\theta \in L_\ell^2(E)$, defined by 
$||\theta||^2:=\langle \theta ,\theta\rangle,$ as the \emph{energy} of $\theta.$  The space of unit flows of bounded energy on $T$ is defined by
\[\mathscr F_{u}^b(T,\ell) :=  \mathscr F_u(T) \cap L_\ell^2(E).\]

The \emph{conductance} $c_{uv}$ of an edge $uv \in E(T)$ is defined as the inverse of the length $\ell(uv)$, i.e., $c_{uv} : = \frac 1{\ell(uv)}$. For $uv\in E(T)$, we define the conductance of the edge with reverse orientation $vu$ by symmetry, $c_{vu} =c_{uv}$.

 Consider the random walk  $RW_{(T,\ell)}$ on the tree $T$ which starts from the root, and which has, for any vertex $v\in V(T)$, a  probability of transition $p_{vu}$ from $v$ to any of its neighbors $u \sim v$ in the tree given by 
 \[p_{vu} = \frac 1{c_{vu}}\,\sum_{\substack{w\in V(T)\\ w\sim v}} c_{vw}.\] 
 In particular, for the standard length function $\ell \equiv 1$, there is an equal chance of moving from a vertex $v$ to any of its neighbors, and $RW_{(T, 1)} = RW(T)$ is the simple random walk on  $T$.
  
  \medskip
  
 Recall that a random walk on a tree is called \emph{transient} if, almost surely, the walk returns to the root only a finite number of times.   Otherwise, it is called \emph{recurrent}. We call a pair $(T, \ell)$ \emph{transient} (resp. \emph{recurrent}) if the random walk $RW_{(T,\ell)}$  is transient (resp. recurrent). 
 The following classical theorem gives a necessary and sufficient condition for the transience of
the  random walk $RW_{(T,\ell)}$, see e.g.~\cite{DS, Lyons83, NW59, Royden}.
\begin{thm}
 The random walk $RW_{(T, \ell)}$ on $T$ is transient if and only if $\mathscr F_{u}^b(T, \ell) \neq \emptyset.$
\end{thm}

Suppose from now on that $(T, \ell)$ is a transient pair, so that we have $\mathscr F_{u}^b(T,\ell) \neq \emptyset$. 
This implies  the existence and uniqueness of a flow of minimum energy $\eta \in \mathscr F_{u}^b(T,\ell)$, c.f.~\cite{Lyons90, DS}. 
 The flow $\eta$ is called the \emph{unit current flow} on $T$. Although not necessary for what follows, we recall the following probabilistic interpretation of $\eta$: for any edge $uv\in E(T)$, $\eta(uv)$ 
is the expected \emph{net} number of crossing of the edge $uv$ for the random walk $RW_{(T, \ell)}$, where net means that a crossing of an edge $uv \in E(T)$  is counted with positive sign while the walk crosses the edge from $u$ to $v$, and with negative sign if the edge is crossed from
$v$ to $u$.

\medskip

The Borel measure associated to the unit current flow $\eta$ is called the \emph{harmonic measure} on $(T,\ell)$ and is denoted 
by $\mu_{\mathrm{har}}$. We recall the following useful property for the harmonic measure whose proof 
can be found e.g. in~\cite{Lyons90, LP}

\begin{prop}\label{prop:height} Let $(T,\ell)$ be a pair consisting of a locally finite tree $T$ with a length function $\ell$. Suppose that the random walk $RW_{(T, \ell)}$ is transient, and let $\eta$ and $\mu_{\mathrm{har}}$
be the corresponding unit current flow and harmonic measure, respectively. Let $\theta$ be any flow of bounded energy in $\mathscr F_u^b(T, \ell)$, and let $\mu_\theta$ be the Borel
measure associated  to $\theta$. Then we have 
\[\lim_{\substack{v\in s \\ v \to \infty}} \ell_\eta([\r,v]) = ||\eta||^2 \qquad \textrm{$\mu_\theta$-a.s. on $\partial T$},\]
where $\ell_\eta([\r,v]) = \sum_{e\in [\r,v]} \eta(e) \ell(e)$.
\end{prop}
In other words, $\mu_{\theta}$-almost surely, the infinite rays in $T$ have the same $\ell_\eta$-length, equal to the energy of the unit current flow $\eta$.

\medskip

We are now ready to state the main theorem of this section on the growth of the factorial sequence associated to the transient pairs $(T, \ell)$. Let $\Gamma$ be the metric tree associated to $(T, \ell)$. We call $\Gamma$, and also $(T, \ell)$, \emph{weakly complete} if it verifies the property that for any vertex $v$ of $T$, any infinite strict path $P$ of $T_v$ has length infinite $\ell$-length, i.e., $\ell(P)=\infty$.  Note that if the values of the length function are $\epsilon$-away from zero for some $\epsilon>0$, e.g., for integer valued length functions such as the standard length function $\ell \equiv 1$, the pair $(T, \ell)$ is automatically weakly complete.

\medskip

Let $(T,\ell)$ be a transient pair. Consider a weighting sequence $\omega_n$ 
for the edges as in Section~\ref{sec:def}, and denote by $T_n =T_{\omega_n}$ the corresponding sequence of factorial trees. 
For each integer $n$, denote by $\widetilde \omega_n$ the \emph{normalized weight function} on $T_n$ defined by 
\[\forall e\in T_n\,\qquad \widetilde \omega_n (e) := \frac 1n \omega(e),\]
that we extend by zero to all the edges in $E(T)\setminus E(T_n)$. Note that for all internal vertices $u$ of the tree $T_n$, we have the flow property at $u$ 
for $\widetilde \omega$ 
$$\sum_{uv\in E(T_n)}\widetilde \omega_n(uv) = \widetilde \omega(\cev uu).$$
In addition, at root $\r$ we have
\[\sum_{\r v\in E(T_n)}\widetilde \omega_n(\r v) =1.\]
We can rephrase this by saying that $\widetilde \omega_n$ is a \emph{partial} unit follow on $T$. 

\begin{defn}\rm
 A measure $\mu_n$ on the extended boundary $\widetilde \partial T$ is called consistent with $\widetilde \omega_n$ 
if for all internal vertex $v$ of $T_n$ we have 
$\mu_n(B_v) =\widetilde w_n(\cev vv)$. 
\end{defn}

In particular, for a choice of elements $\rho_0, \rho_1, \dots$ in the extended boundary $\widetilde \partial T$ in the definition of the factorial sequence in Section~\ref{sec:locfin-intro}, the discrete averaging measures 
$\mu_n = \frac 1n \sum_{j=0}^{n-1} \delta_{\rho_j}$ are consistent with $\widetilde \omega_n$. 

Obviously, $\widetilde \omega_n$ depends on the choices we made at each step in constructing the 
factorial sequence. However, the following theorem shows, when the pair $(T, \ell)$ is weakly complete and transient, 
asymptotically, the behavior of $\widetilde \omega_n$ is independent of the 
choices. More precisely,

\begin{thm}\label{thm:measure} Let $(T, \ell)$ be a transient pair, and let $\Gamma$ be the corresponding metric tree. Denote by $\eta$ and $\mu_{\mathrm{har}}$ the unit current flow on $T$ and the harmonic measure on $\partial T$, respectively. Assume that $\Gamma$ is weakly complete.  Then
\begin{itemize}
 \item[(1)] the sequence $\widetilde \omega_n$ converges point-wise to $\eta$, i.e., for any edge $e$, 
 we have  
 \[\lim_{n \to \infty} \widetilde \omega_n(e) =\eta(e).\]
 \item[(2)] for any sequence of measures $\mu_n$ on $\partial T$ with $\mu_n$ 
 consistent with $\widetilde \omega_n$, the sequence $\mu_n$ converges weakly to the harmonic measure 
 $\mu_{\mathrm{har}}$.
 \item[(3)] the (logarithmic) factorials of $(T, \ell)$ satisfy the following asymptotic
 \[H(\Gamma) = \lim_{n \rightarrow \infty} \frac 1n \, n!_\Gamma = ||\eta||^2.\]
\end{itemize} 
\end{thm}

\begin{remark}\label{rem:ex}\rm The condition of being weakly complete is necessary as the following example shows. Consider a pair $(T, \ell)$ with $T$ a rooted tree with $\r$. Assume $\r$ has two children $u,v$, $T_u$ is an infinite path of finite $\ell$-length, and $(T_u,\ell)$ is a recurrent pair. Then the unit current flow on $(T,\ell)$ is the unit flow on the strict path $P_u$ which contains $u$. However, the limit of $\widetilde \omega_n(\r u)$ is obviously zero, as $\omega_n(\r u) =1$ for all large $n$. This is a typical situation where, in absence of the weakly completeness assumption, the arguments of the next section fail.
\end{remark}

The rest of this section is devoted to the proof of Theorem~\ref{thm:measure}. 

\subsection{Upper bound on the growth of factorials}
 
 In this section, we assume $T$ is a locally finite tree and $\ell$ is a length function on $T$ so that the pair $(T, \ell)$ is  transient, so $\mathscr F^b_u(T) \neq \emptyset$, and $(T, \ell)$ is weakly complete. Both the condition are necessary for what follows. 

Let $\omega_n$ be a sequence of weightings resulting in the construction of the factorial sequence of $(T, \ell)$.
Let $T_n = T_{\omega_n}$, and denote by $U_n$ the set of all the unsaturated vertices of $T_n$, as before.  Let 
$\widetilde \omega_n =\frac 1n \omega_n$, that we extend by zero to all the edges $E(T) \setminus E(T_n)$. 

Let $\theta \in \mathscr F^b_{u}(T, \ell)$ be a unite flow of bounded energy on $T$. We have the following proposition.

\begin{prop}\label{ub-growth}
 For each non-negative integer $n$, there exists a vertex $v \in U_n$ such that for all the edges $e$ on the oriented path $[\r,v]$ from $\r$ to $v$, we have 
 \begin{equation}\label{prop:p1}
 0<\widetilde \omega_n(e)\leq \theta(e).
 \end{equation}
  In particular, the path $P$ is part of an infinite ray of $T$, and we have 
 \begin{equation}\label{prop:p2}
 H(T, \ell)\leq \sum_{e\in P} \theta(e) \ell(e).
 \end{equation}
\end{prop}

As an application of this proposition, we get the following interesting corollary.

\begin{cor}\label{cor:ub} For any transient and weakly complete pair $(T, \ell)$, we have 
\[\lim_{n\to \infty} \frac 1n (n!_{(T, \ell)}) \leq \|\eta\|^2 < \infty,\]
where $\eta \in \mathscr F^b_u(T, \ell)$ is the unite current flow. \end{cor}
 \begin{proof} Take $\theta = \eta$ in Proposition~\ref{ub-growth}, for the unite current flow $\eta$ on $(T,\ell)$. By Proposition~\ref{prop:height}, $\mu_{\mathrm{har}}$-almost surely, all the infinite rays of $T$ have the same length, equal to $\|\eta\|^2$, with respect to $\ell_\eta$. For the path $P$ in the proposition, since $\theta$ is positive on any edge of $[\r, v]$, we get
 $\sum_{e \in P} \eta(e)\ell(e) \leq \|\eta\|^2$, and the corollary follows.  
 \end{proof}

\begin{proof}[Proof of Proposition~\ref{ub-growth}]
 We construct the path $P$ proceeding by induction and using a greedy procedure. We actually prove both the statements in the proposition simultaneously.   Note that for all edges $e$ in $E(T_n)$, we automatically have $\widetilde \omega_n(e) >0$, by the definition of $T_n$.

 Let $n\in \mathbb N$.  Consider first the equation
 \begin{equation*}
 (*)\qquad \sum_{\r u\in E(T_n)} \omega_n(\r v)  = n = n\sum_{\r u \in E(T)} \theta(\r u).
 \end{equation*}
 On of the two following cases, (0,1) or (0.2), can happen. 
 
 \medskip
 
 (0.1) Either, there exists an edge $\r u\in E(T)$ which does not belong to $T_n$ and which satisfies $\theta(\r u) >0$.
In this case, we have $\widetilde \omega_n(\r u) =0$. Let $v=\r$. Then the path $[r,v]$ is reduced to a single vertex $\r$, and Inequality~\eqref{prop:p1} trivially holds. In addition, since $\theta(\r u)>0$, $v$ must be part of an infinite path in $T$, and we have 
$n!_{(T, \ell)} =0 \leq \sum_{e\in P}\theta(e)\ell(e)$, so that Inequality~\eqref{prop:p2} holds as well.

\medskip
(0.2) Otherwise, we have $\theta(\r v)=0$ for all $\r v \in E(T) \setminus E(T_n)$, and Equation~(*) gives
 \[\sum_{\r u\in E(T_n)} \omega_n(\r v)  = n = n\sum_{\r u \in E(T_n)} \theta(\r u).\]
 Therefor there exists an edge $\r z_1\in E(T_n)$ with $0<\widetilde \omega_n(\r z_1) \leq \theta(\r z_1)$. 
 Let $P_1$ be a strict path in $T$ which contains $\r z_1$. Since $\theta(\r z_1)>0$, and the norm of 
 $\theta$ is finite, the strict path $P_1$ has to be of  finite $\ell$-length, i.e., $\ell(P_1)<\infty$. By the assumption that $(T, \ell)$ is weakly complete,  the path $P_1$ has to be a finite path in $T$. Denote by 
 $u_1$  the other end-vertex of $P_1$. Note that $u_1$ is a vertex of $T_n$ which is not a leaf of $T$, since, otherwise, we should have $\theta(e) =0$ for all the edges in $P_1$.
 
 \medskip

Proceeding inductively on $k\in \mathbb N$, assume that we have a sequence of vertices 
$u_0=\r, u_1, u_2, \dots, u_{k}$ and $z_1, z_2, \dots, z_{k}$ such that there is a strict oriented path $P_i$ in $T_{u_{i-1}}$
from  $u_{i-1}$ to $u_{i}$ which  contains the edge $u_{i-1}z_{i} \in E(T_n)$,  for all $1\leq i\leq k$, 
the vertex $u_k$ is not a leaf of $T$, and $\theta(e) \geq \widetilde \omega_n(e)>0$ for all edges in any path among the $P_i$s. 
One of two following cases can happen

\medskip

($k$.1) either there exists an edge $u_ku \in E(T)$ 
 which does not belong to $T_n$ so that  $\theta(u_ku)>0$. In this case, we let $v=u_k$ and let $P = [\r,u_k]$ (the union of all the paths $P_1, \dots, P_{k}$). Since $\theta(u_1u)>0$, the path $P$ is part of an 
 infinite path which contains the edge $u_1u$, and we have 
 \[\frac 1n n!_T \leq \sum_{e\in P} \widetilde \omega_n(e) \leq \sum_{e\in P} \theta(e), \]
 which proves the result. 

 \medskip
 
 ($k$.2) Otherwise, for all the edges $u_ku \in E(T) \setminus E(T_n)$, we must have $\theta(u_ku)=0$. Since $\theta(e) >0$ for any $e\in [\r, u_k]$, and $\theta$ is a flow, this implies that $u_k$ is not a leaf of $T_n$.
Therefore, we must have
 $$\sum_{u_ku\in E(T_n)}\widetilde \omega_n(u_ku) = \widetilde 
 \omega_n(\cev u_k u_k)  = \widetilde \omega_n(u_{k-1}z_{k-1}) \leq \theta(u_{k-1}z_{k-1}) = \theta(\cev u_ku_k) = \sum_{u_ku\in E(T_n)}\theta(u_k u).$$
In particular, there exists $u_kz_{k+1}\in E(T_n)$ with 
\[0<\widetilde \omega_n(u_kz_{k+1}) \leq \theta(u_kz_{k+1}).\]
Let $P_{k+1}$ the strict path in $T_n$ starting from $u_k$ which contains the edge $u_kz_{k+1}$. Since $\theta$ has bounded norm, and the value of $\theta$ on all edges of 
$P_{k+1}$ are equal to $\theta(u_kz_{k+1})>0$,  from the assumption that  $(T, \ell)$ is weakly complete, we infer that 
the path $P_{k+1}$ is finite in $T$. Let $u_{k+1}$ be  the other end-point of 
$P_{k+1}$, and note that $u_{k+1}$ is not a leaf in $T$.

Since the tree $T_n$ is finite, this process eventually stops, i.e., there is an $m$ such that 
the case $(m.1)$ happens, and the proposition follows.
 \end{proof}

\subsection{Point-wise convergence of $\widetilde \omega_n$ in the case $H(T, \ell)$ is finite.}
In this section, we assume that the pair $(T, \ell)$ is so that  the sequence 
$\frac 1n n!_{(T, \ell)}$ converges to a finite number $H = H(T, \ell)<\infty.$ In particular, by Corollary~\ref{cor:ub}, what follows applies to transient weakly complete pairs $(T, \ell)$.  Our main result is Theorem~\ref{thm:pwconv} which shows that for any edge $e$, the sequence $\widetilde \omega_n(e)$  converges to a number $[0,1]$.

\medskip

Without loss of generality, using Claim~\ref{claim2.6}, we can assume that the root has branching $\br(\r)=d \geq 2$. 

Denote by $u_1, \dots, u_d$ all the children of $\r$, and set $N_j = N_{T_{u_j}}$. Let $S_j = \Bigl\{a^j_i\Bigr\}_{0\leq i<N_{j}}$ be the factorial sequence of the subtree $T_{u_j}$ (which might be finite), and define the sets $A_, \dots, A_d$ as in the previous section
$A_j = \Bigl\{ a^j_i + i \ell_{\r u_j}\Bigr\}_{0\leq i < N_j}.$ 
 
 Enumerate the terms in the multiset union $A$ of the sets $A_j$ in a fixed increasing order induced by the weighting sequence $\omega_n$, depending on to which subtree $T_{u_j}$ the (factorial-determining) vertex $x_n$ in the definition of the factorials belongs. So each element of $A$ is labeled with an index among $1,\dots, d$. For each $j$, denote by $k_j(n)+1$ the number of elements among the first $n+1$ terms in $A$ labeled by $j$, i.e., the number of indices $0\leq i < N_j$ with $a_i^j+i \ell_{\r u_j}$ among the first $n+1$ terms of the sequence $A$. 
 We have the following  straightforward, but useful, inequalities
\[0\leq k_j(n) - k_j(n-1) \leq 1,\]
\begin{equation}\label{eq2}
a_{k_j(n)}^j+k_j(n) \ell_{\r u_j} \leq n!_{(T, \ell)} \leq a_{k_j(n)+1}^j+ \bigl(k_j(n)+1\bigr) \ell_{\r u_j},
\end{equation}
for all $j=1, \dots, n$, and all $0\leq n<N_j$ .
Note that in addition,  by exhaustiveness of the weighting process proved in Proposition~\ref{prop:cover}, we have 
\begin{center}
\emph{the tree $T_{u_j}$ is infinite if and only if $k_j(n) \to \infty.$}
\end{center}
For an integer $1\leq j\leq d$ with $N_{j} = \infty$, denote by $H_{u_j}$ the limit 
\[H_{u_j}:=\lim_{k\to \infty} \frac1k a_{k}^j = H(T_{u_j}, \ell_{|T_{u_j}}),\]
which exists by Corollary~\ref{cor:limit}, and which belongs to the interval $(0,\infty]$. 

\medskip

For all $1\leq j \leq d$, and all non-negative integers $n<N_j$, we get from Equation~\eqref{eq2}
\begin{align}\label{eq:lim2}
\frac {k_j(n)}n  \,\Bigl(\,\frac {a_{k_j(n)}^j}{k_{j}(n)}+\ell_{\r u_j}\,\Bigr)\, \leq \,\frac 1n n!_{(T,\ell)} \leq  \frac{k_j(n)+1}n \Bigl(\,\frac{a_{k_j(n)+1}^j}{k_{j}(n)+1}+\ell_{\r u_j}\,\Bigr)
\end{align}
Note that we have $\frac {k_j(n)}n = \widetilde \omega_n(\r u_j)$.
We distinguish the following three different cases:

\begin{itemize}
\item[(1)] We have $N_j <\infty$. In this case, since $k_j(n)< N_j$, we get $\widetilde \omega_n(ru_j) \to 0$ as $n$ tends to infinity. 
\item[(2)] We have $N_j =\infty$ and $H_{u_j}<\infty$. In this case, making $n$ tend to infinity, we get from \eqref{eq:lim2} that the limit of $\widetilde \omega_n(e_j)$ exists and is equal to
\[\lim_{n\to \infty} \widetilde \omega_n(e_j) = \frac H{H_{u_j}+ \ell_{\r u_j}}>0.\]

\item[(3)] We have $N_j=\infty$ and $H_{u_j}=\infty$. In this case, since $\frac 1n n!_{(T,\ell)}$ converges to a finite $H$, and the term $a^j_{k_j(n)}/k_j(n)$ in~\eqref{eq:lim2} converges to infinity, we must have
\[\lim_{n\to \infty} \widetilde \omega_n(e_j) = \frac{k_j(n)}n = 0.\]
 \end{itemize}

We have thus proved the point-wise convergence of $\widetilde \omega_n$ at all the pending edges at the root of $T$. Since $\widetilde \omega_n$ is a partial flow, and the weighting is exhaustive, it follows that in the first and third cases, we actually have $\lim_{n\to \infty} \widetilde \omega_n(e) =0$ for all the edges $e$ in the subtree $T_{u_j}$.

Note that, again by the exhaustiveness of the weighting, and since $\widetilde \omega_n$ is a partial flow, we get the equation
\begin{align}
H \sum_{j=1}^d \frac 1{H_{u_j}+\ell_{\r u_j}} =1.
\end{align}

\begin{defn}\rm \label{def:hv}
For any vertex $v \in V(T)$, define $H_v \in [0,\infty]$ as follows:
\[H_v:=\begin{cases}
0 & \qquad \textrm{if the tree $T_v$ is finite,}\\
H(T_v, \ell_{|T_v}) = \lim_{n\to \infty}\frac 1n (n!)_{(T_v,\ell_{|T_v})} & \qquad \textrm{otherwise.} 
\end{cases}.
\]
\end{defn}
Proceeding now by induction on the generation $|v|$ of vertices $v\in T$, we prove the following theorem.
\begin{thm}[Point-wise convergence of $\widetilde\omega_n$]\label{thm:pwconv} For any edge $uv \in E(T)$, the limit,  when $n$ tends to infinity, of $\widetilde w_n(uv)$ exists. It is non-zero precisely when $H_v $ lies in the interval $(0, \infty)$, in which case, the limit is given by 
\[\lim_{n\to \infty} \widetilde \omega_n(uv) = \prod_{\substack{w\in [\r,v] \\ w\neq \r}} \frac {H_{\cev w}}{H_{w} + \ell_{\cev w w}}.\]
\end{thm} 

The rest of this section is devoted to the proof of this theorem. We proceed by induction on the generation $|v|$ of $v$. By what proceeded the statement of the theorem, we already proved the theorem for the pending edges at $\r$, i.e., in the case $u=\r$ and $|v|=1$. 

Let $m\in \mathbb N$, and assume that the statement holds for all vertices $v$ with $|v|=m$. We prove the theorem for all vertices of generation $m+1$. So let $uv\in E(T)$ with $|v|=m+1$. Since $|u|=m$, the statement already holds for the edge $\cev u u \in E(T)$. 
Two cases can happen:

\begin{itemize}
\item Either, $\lim_{n \to \infty} \widetilde \omega_n(\cev u u) = 0$. In this case, for all edges in the subtree $T_{\cev u}$ we have $\widetilde \omega_n(e) \leq \widetilde \omega_n(\cev uu)$, and so $\lim_{n\to \infty}\widetilde \omega_n(e)=0$. In particular, the limit when $n$ tends to infinity of $\widetilde \omega_n(uv)$ exists and is equal to zero.
\item Or, $\lim_{n \to \infty} \widetilde \omega_n(\cev u u) >0$. \end{itemize}
In this case, we have $\lim_{n\to\infty} \omega_n(\cev u u) =\infty$, and by the hypothesis of the induction,
we have $H_{\cev u} \in (0, \infty)$ and the following equation holds:
\[\lim_{n\to \infty} \widetilde \omega_n(\cev u u) = \prod_{\substack{w\in [\r, u] \\ w\neq \r}} \frac {H_{\cev w}}{H_{w} + \ell_{\cev w w}}.\]
In particular, the tree $T_{\cev u}$ is infinite. 

Denote by $I_u = \Bigl\{p_0, p_1,p_2,\dots \Bigr\} \subset \mathbb N$ the set of all the non-negative integers $n$ where an unweighted edge incident to a vertex of the subtree $T_{u}$ is weighted in the description of the weighting process, enumerated in an increasing order, so $p_0<p_1<\dots$. 
The weighting sequence $\omega_{p_j}$ restricted to the subtree $T_u$ produces a weighting sequence $\omega_j^u$ for $(T_u, \ell_{|T_u})$. Since $0<H_u<\infty$, by what proceeded before the statement of the theorem applied to $T_u$, we get that
\begin{itemize}
\item for all edges $uv \in E(T_u)$, the limit when $j$ tends to infinity of $\widetilde \omega^u_j =\omega_j^u/j$ exists; and
\item this limit is non-zero precisely when $H_v \in (0, \infty)$, in which case, the limit is given by 
\[\lim_{j \to \infty} \widetilde \omega_j^u (uv) = \frac{H_u}{H_v+ \ell_{uv}}.\]
\end{itemize}
We now observe that for all edges $uv\in E(T)$,
\begin{align*} 
\lim_{n \to \infty } \widetilde \omega_n(uv) &= \lim_{n \to \infty }\frac 1n \omega_n(uv) = \lim_{n \to \infty } \Bigl(\frac{\omega_n(\cev u u)}n \cdot \frac {\omega_n(uv)}{\omega_n(\cev u u)}\Bigr) \\
&= \lim_{n\to \infty} \widetilde \omega_n(\cev u u) \,. \lim_{n\to \infty} \widetilde \omega_j^u(uv).
\end{align*}
Combing all these together, we finally get that for $uv\in E(T)$,
\begin{itemize}
\item the limit when $n$ tends to infinity of $\widetilde \omega_n(uv)$ exists; and 
\item it is non-zero precisely when $\widetilde \omega_j^u(uv)$ has a non-zero limit when $n$ tends to infinity, i.e., when $H_v \in (0, \infty)$, in which case we have 
\begin{align*}
\lim_{n \to \infty } \widetilde \omega_n(uv)  &= \Big(\prod_{\substack{w\in [\r, u] \\ w\neq \r}} \frac {H_{\cev w}}{H_{w} + \ell_{\cev w w}}\Bigr)\cdot \frac {H_u}{H_v+ \ell_{uv}}= \prod_{\substack{w\in [\r, v] \\ w\neq \r}} \frac {H_{\cev w}}{H_{w} + \ell_{\cev w w}}\,\,.
\end{align*}
\end{itemize}
This finishes the proof of our theorem. \qed

\subsection{Equivalence of the transience of $(T, \ell)$ with finiteness of $H(T, \ell)$}
In this section, we prove Theorem~\ref{equivalence-intro}. So let $(T, \ell)$ be a pair consisting of an infinite locally finite tree and a length function $\ell$ on $T$. Assume that $(T, \ell)$ is weakly complete. 
\begin{thm}The following two statements are equivalent.
\begin{itemize}
\item[$(i)$] The pair $(T, \ell)$ is transient.
\item[$(ii)$] We have $H(T, \ell)<\infty$.
\end{itemize}
\label{thm:equiv} \end{thm}
The implication $(i) \Rightarrow (ii)$ is already proved in Corollary~\ref{cor:ub}. We prove $(ii)$ implies $(i)$.

Let $\omega_n$ be a sequence of weighting for the pair $(T, \ell)$. Assume that $H(T, \ell) <\infty$. For each vertex $v\in T$, define $H_v$ by Definition~\ref{def:hv}. By the results of the previous section, 
we have the point-wise convergence of the sequence $\widetilde \omega_n$ to some $\phi : E(T) \rightarrow \mathbb R_{\geq 0}$. Obviously, we have $\phi \in \mathscr F_u(T)$, and by what we proved in the previous section, the non-zero values of $\phi$ on edges are given by 
\[\forall \,\, uv\in E(T) \textrm{ with }\phi(uv)\neq0 , \qquad \phi(uv)  =\prod_{\substack{w\in [\r,v]\\ w\neq \r}} \frac {H_{\cev w}}{H_{w} +\ell_{\cev w w}}.\]
The following claim finishes the proof of our theorem.
\begin{claim}\label{claim4.10} The unit flow $\phi$ has bounded energy, and thus belongs to $\mathscr F_u^b(T, \ell)$.  
\end{claim}
 Define the function $F: V(T) \rightarrow \mathbb R$ as follows. Let $F(0) =0$, and for all vertices $v\in V(T) \setminus \{\r\}$, define
 \[F(v):= \sum_{\substack{w \in [\r,v]\\ w\neq r}} \phi(\cev w w) \ell_{\cev w w}.\]
\begin{claim} We have for all $v\in  V(T)$, $F(v) \leq H(T, \ell)$.
\end{claim}
\begin{proof} It will be enough to prove the result for any vertex $v \in V(T)$ with $0<H_v <\infty$. Let $v$ be  such a vertex and denote by $v_0=\r, v_1, \dots, v_k=v$ all the vertices on the path  $[\r,v]$ from $\r$ to $v$, with $e_i :=v_{i-1}v_i \in E(T)$ for $i=1, \dots, k$. We have
\begin{align*}
F(v)& = \sum_{\substack{u\in [\r,v] \\ u \neq \r}} \phi(\cev u u)\ell(\cev u u) = \sum_{j=1}^k \ell(e_j)\prod_{i=1}^j \frac{H_{v_{i-1}}}{H_{v_i} + \ell(e_i)} \\
&= \sum_{j=1}^k \Bigl(H_{v_j}+\ell(e_j) - H_{v_j}\Bigr)\prod_{i=1}^j \frac{H_{v_{i-1}}}{H_{v_i} + \ell(e_i)} \\
&= \sum_{j=1}^k \Bigl(\bigl(H_{v_j}+\ell(e_j)\bigr)\prod_{i=1}^j \frac{H_{v_{i-1}}}{H_{v_i} + \ell(e_i)} -   H_{v_j}\prod_{i=1}^j \frac{H_{v_{i-1}}}{H_{v_i} + \ell(e_i)}\Bigr)\\\
&= H_\r - H_\r\cdot \frac{\prod_{j=1}^{k}H_{v_j}}{\prod_{j=1}^{k}(H_{v_j} + \ell(e_j))} \leq H_\r = H{(T, \ell)}.
\end{align*} 
\end{proof}

By the previous claim we can extend $F$ to a function on the boundary $\partial T$. We have
\begin{claim} Denote by $\mu_\phi$ the measure of mass one on $\partial T$ associated to $\phi$. 
We have $\|\phi\|^2 = \int_{\partial T} F \,d\mu_\phi.$
\end{claim}
\begin{proof} This is  standard fact, and can be found e.g. in~\cite[Section 4]{Lyons90}. The idea is that one can write 
\[\sum_{uv\in E(T)} \phi(uv)^2 \ell(uv) = \sum_{uv\in E(T)} \phi(uv) (F(v)-F(u)).\]
Consider a finite cut set $C$ of $T$ with vertex set $U$ and with complementary vertex set $W = V(T) \setminus U$.  Using that $\phi \in \mathscr F_u(T)$,  we have for the partial sum
\[\sum_{\substack{u\in U \\ uv\in E(T)}} \phi(uv) (F(v)-F(u)) = F(\r)+ \sum_{v\in \partial W} F(v) \phi(\cev v v)  = \sum_{v\in \partial W} F(v) \mu_\phi(B_v),\]
where $\partial W$ is the set of all vertices $v$ with $\cev v \in U$, and $B_v$ is the open subset of $\partial T$ defined previously. The result now follows by tending the cut set $C$ to infinity. 
\end{proof}

Combining the two previous claims gives 
\begin{equation}\label{eq:ub}
\|\phi\|^2  = \int_{\partial T} F d\mu_\phi \leq H(T, \ell)<\infty,
\end{equation} and finishes the proof of Claim~\ref{claim4.10}. The proof of Theorem~\ref{thm:equiv} is now complete.

\subsection{Proof of Theorem~\ref{thm:measure}}
With what we proved in the previous sections, we can now complete the proof of Theorem~\ref{thm:measure}. 

Let $(T, \ell)$ be a transient pair, and denote by $\eta \in \mathscr F_u^b(T, \ell)$ the corresponding unit current flow on $T$. Denote by $\phi$ the point-wise limit of $\widetilde \omega_n$ for a weighting sequence $\omega_n$. 

By Corollary~\ref{cor:ub}, we have $H(T, \ell) \leq \|\eta\|^2$. On the other hand, by Inequality~\ref{eq:ub}, 
we have $\|\phi\|^2 \leq H(T,\ell)$. Since $\eta \in \mathscr F_u^b(T, \ell)$ is the flow of minimum energy, it follows that $\phi = \eta$, which is part $(1)$ of  Theorem~\ref{thm:measure}. Part $(2)$  is a direct consequence of part $(1)$. Part $(3)$ follows from the equality $\phi=\eta$ combined with the inequalities of Corollary~\ref{cor:ub} and Equation~\ref{eq:ub}.

\section{Concluding remarks} \label{sec:concluding}
We include here a  brief discussion of some results and questions complementary  to what we presented in the previous sections. 
\subsection{Removed version}
Let $t\in \mathbb N$. Let $T$ be a locally finite tree, $\ell$ and $\chi$ a length and capacity function on $T,$ respecitvly. One can define a $t$-removed version of the factorials associated to $(T, \ell, \chi)$. Choose $\rho_0, \dots, \rho_{t-1} \in \widetilde T$ arbitrarily, in such a way that the capacity condition is verified. Assuming that $\rho_0, \dots, \rho_{n-1}$ are chosen, one chose $\rho \in \widetilde \partial T$ among those unsaturated elements $\rho \in \widetilde \partial T$ which minimizes the quantity 
\[a_n(\rho)=\min_{\substack{A \subset \{0, \dots, n-1\}\\ |A| = n-t}} \sum_{j\in A} \langle \rho, \rho_j\rangle,\]
and define $a_n = a_n(\rho_n)$.

\begin{thm} The sequence $\{a_n\}$ only depends on $(\Gamma, \chi)$ and $t$.
\end{thm}
The proof is similar to the proof of Theorem~\ref{thm:main1}, and leads to a combinatorial proof  of a generalization of~\cite{Bha09}.  Define $n!_{(\Gamma, \chi)}^{\{t\}}:=a_n$. We have the following theorems.
\begin{thm} Let $(T, \ell)$ be a pair of a locally finite tree $T$ rooted at $r$, and a length function $\ell$ on $T$. Let $d=\br(r)$ and denote by $u_1, \dots, u_d$ all the children of $r$. Let $\Gamma_j$ be the metric tree associated to the pair $(T_{u_j}, \ell_{|T_{u_j}})$, and let $\chi_j$ be the restriction of $\chi$ to $T_{u_j}$. We have for all $0\leq n < N_{T, \chi}$,
\[n!^{\{t\}}_{(\Gamma, \chi)}  = \min_{\substack{ (t_1, \dots, t_d)\in \mathbb N_*^d \\ t_1+\dots +t_d=t}}
 \,\, \min_{\substack{(n_1, \dots, n_d) \in \mathbb N_*^d\\
n_1+ \dots +n_d = n+1}} \max\,\, \Bigl\{\,(n_j-1)!^{\{t_j\}}_{(\Gamma_j, \chi_j)} + (n_j -1) \ell(\r u_j)\,\Bigr\}_{j=1}^d.\]
\end{thm}

\begin{thm}\label{thm:limt}
For any pair $(T, \ell, \chi)$ with $N_{T,\chi} =\infty$, and any $t\in \mathbb N$, we have 
\[\lim_{n\to \infty} \frac 1n n!^{\{t\}}_{(T, \ell, \chi)}   =  \lim_{n\to \infty} \frac 1n n!_{(T, \ell,\chi)}.\]
\end{thm}

\subsection{Subsets of $\mathbb Z$ versus $p$-trees} 
Let $\mathcal P$ be the set of prime numbers in $\mathbb Z$. 
For any subset $X \subset \mathbb Z$ of integers, and any prime $p$, denote by $T_{X, p}$ the tree associated to $X \subset \mathbb Z \subset \mathbb Z_p$. By the discussion in the introduction, we have
\[n!_X  = \prod p^{n!_{T_{X,p}}},\]
where $n!_X$ denotes the Bhargava's factorial of $n$ for the set $X$.
Note in particular that for two subsets $X, Y \subset \mathbb Z$, we have $n!_X = n!_Y$ if and only if $n!_{T_{X,p}} = n!_{T_{Y, p}}$ for all $p\in \mathcal P$, and any $n\in \mathbb N$. In particular, two subsets of integers with the same $p$-trees, have the same factorials.
 
In this regard, it seems natural to wonder $(1)$ how can two subsets of the integers have the same collection of $p$-trees ? $(2)$ what can be said about the $p$-trees of two subsets $X$ and $Y$ if they have same factorials? and $(3)$ which collections of $p$-trees, one for each $p \in \mathcal P$, come from a subset $X$ of $\mathbb Z$ ? The following proposition shows that in general two sets $X$ and $Y$ with the same factorial sequence can be very different.
\begin{prop} Let $\epsilon\in (0,1)$ be a fixed positive real number.
 Let $X$ be a random subset of $\mathbb Z$ obtained by choosing 
 any integer $k \in \mathbb Z$ with probability $\epsilon$ independently at random.
 For all $n \in \mathbb N$, we have $n!_X = n!$.
\end{prop}
\begin{proof}
For any prime $p$, the tree $T_{X,p}$ is the regular $p$-tree $\mathscr T_p$. It follows that $n!_X = n!.$ 
\end{proof}

In particular, it seems very unlikely to have an answer to (1) without any further assumption on
$X$ and $Y$.

Regarding (3), by applying Chinese reminder lemma, we have the following proposition.
\begin{prop} \label{prop:crm}
 Let $\{T_p\}$ be a sequence of trees one for each prime $p$, with $T_p$ a subtree of the $p$-regular tree $\mathscr T_p$ without any leaf. Assume there exists an integer $n$ such that for all primes
 $p>n$, the tree $T_p$ is the complete $p$ tree. There exists a subset $X$ of $\mathbb Z$ whose associated $p$-tree is equal to $T_p$ for all prime $p$.
\end{prop}
Question (3) for more general collection of trees seems to be quite interesting on its own.

Regarding question $(2)$ above, applying the above proposition, we infer from the existence of non-isomorphic trees with the same factorial sequence for a prime $p$, the existence of two subsets  $X$ and $Y$ with the same factorial sequence and without necessary the same adelic trees. So again the answer to $(2)$ seems to be rather delicate. 
\subsection{Factorials of definable sets} A structure theorem for definable sets over $p$-adic numbers is proved by Halupczok in~\cite{Hal10, Hal14}, see also~\cite{CCL12}. It appears to be an interesting problem to study the factorials of trees of definable sets.

\subsection*{Acknowledgement}  
Some of the ideas and results of this paper were presented  to  A.
Prodhomme and F. Reverchon  during the winter semester in 2014, when I was
supervising their first year memoir work at ENS on Bhargava's work on
rings of integer valued polynomials. I would like to thank them both for
their interest in the subject and for the discussions during that
semester. It is also a pleasure to thank A. Salehi Golsefidy, A. Rajaei, and H. Amini for helpful conversations.

\end{document}